\RequirePackage{tikz}
\documentclass{article}

\usepackage{hyperref}

\usepackage{multirow}%
\usepackage{amsmath,amssymb,amsfonts}%
\usepackage{amsthm}%
\usepackage{mathrsfs}%
\usepackage{xcolor}%
\usepackage{textcomp}%
\usepackage{manyfoot}%
\usepackage{booktabs}%
\usepackage{algorithm}%
\usepackage{algpseudocode}%
\usepackage{nicefrac}

\usepackage{pgfplots}
\usepackage{tikz}
\usepackage{mathtools}

\usetikzlibrary{arrows.meta}
\usetikzlibrary{backgrounds}
\usepgfplotslibrary{patchplots}
\usepgfplotslibrary{fillbetween}
\pgfplotsset{%
    layers/standard/.define layer set={%
        background,axis background,axis grid,axis ticks,axis lines,axis tick labels,pre main,main,axis descriptions,axis foreground%
    }{
        grid style={/pgfplots/on layer=axis grid},%
        tick style={/pgfplots/on layer=axis ticks},%
        axis line style={/pgfplots/on layer=axis lines},%
        label style={/pgfplots/on layer=axis descriptions},%
        legend style={/pgfplots/on layer=axis descriptions},%
        title style={/pgfplots/on layer=axis descriptions},%
        colorbar style={/pgfplots/on layer=axis descriptions},%
        ticklabel style={/pgfplots/on layer=axis tick labels},%
        axis background@ style={/pgfplots/on layer=axis background},%
        3d box foreground style={/pgfplots/on layer=axis foreground},%
    },
}

\usepackage{url}

\hypersetup{hidelinks}

\newcommand{\R}{\mathbb{R}}

\newcommand{\PJ}{\cal{P}}

\newtheorem{theorem}{Theorem}
\newtheorem{lemma}[theorem]{Lemma}
\newtheorem{definition}{Definition}
\newtheorem{corollary}[theorem]{Corollary}
\newtheorem{proposition}[theorem]{Proposition}

\newtheorem{remark}{Remark}
\newtheorem{example}{Example}

\newlength{\figurewidth}
\pgfplotsset{compat=1.18,width=\figurewidth,height=0.3\figurewidth,every axis/.append style={axis lines=center,ylabel style={at={(ticklabel cs:0.5)},rotate=90,anchor=south},,xlabel style={at={(ticklabel cs:0.5)},anchor=north}}}
\newcommand{\fnm}[1]{#1}
\newcommand{\sur}[1]{#1}

\newcommand{\email}[1]{}
\date{}

\newcommand{\keywords}[1]{\par\textbf{Keywords:} #1}
\newcommand{\pacs}[1]{\par\textbf{MSC Classification:} #1}

\newcommand{\backmatter}{}
\newcommand{\bmhead}[1]{\section{#1}}

\title{Extragradient method with feasible inexact projection to variational inequality problem}
\author{
\fnm{R.} \sur{D\'iaz Mill\'an} \email{r.diazmillan@deakin.edu.au}\and 
\fnm{O.}\sur{P. Ferreira}  \email{orizon@ufg.br}\and 
\fnm{J.} \sur{Ugon} \email{ j.ugon@deakin.edu.au}}

\begin{document}
\maketitle
\abstract{The variational inequality problem in finite-dimensional Euclidean space is addressed in this paper, and two inexact variants of the extragradient method are proposed to solve it. Instead of computing exact projections on the constraint set, as in previous versions extragradient method, the proposed methods compute feasible inexact projections on the constraint set using a relative error criterion. The first version of the proposed method provided is a counterpart to the classic form of the extragradient method with constant steps. In order to establish its convergence we need to assume that the operator is pseudo-monotone and Lipschitz continuous, as in the standard approach.   For the second version,  instead of a fixed step size,  the method presented finds a suitable step size in each iteration by performing a line search. Like the classical extragradient method, the proposed method does just two projections into the feasible set in each iteration. A full convergence analysis is provided, with no Lipschitz continuity assumption of the operator defining the variational inequality problem.}



\section{Introduction}
This paper addresses the variational inequality problem in finite-dimensional Euclidean spaces. This problem is formally stated as follows: Let  $F: {\mathbb R}^n\to {\mathbb R}^n$ be an operator and   ${\cal C}\subset {\mathbb R}^n$  be a nonempty and closed convex set, the variational inequality problem associated with $F$ and  $ {\cal C}$ (VIP($F,{\cal C}$)) consists in finding a  $x^*\in {\cal C}$ such that
\begin{equation} \label{eq:mp}
 \left\langle F({x^*}),  x-{x^*} \right\rangle\geq 0 , \qquad \forall~x\in {\cal C}.
\end{equation}
We denote by ${\cal C}^*$ the {\it solution set} of  problem~\ref{eq:mp}, which we will assume to be nonempty.  The variational inequality problem has captured the attention of the mathematics programming community, not only due to its intrinsic interest but also because it serves as an abstract model for several families of problems in nonlinear analysis and their applications.  For instance, if  $F=\nabla f$, where  $f:{\mathbb R}^n\to {\mathbb R}$ is a differentiable function, then the VIP($F,{\cal C}$) corresponds to the problem of minimizing the function $f$ constrained to the set ${\cal C}$.  When ${\cal C}$ is a cone ${\cal K}\subseteq {\mathbb R}^n$, the  VIP($F,{\cal C})$ is a complementary problem, which is stated in the following form: Compute  $x^*\in {\cal K}$ such that, $F(x^*)\in {\cal K}^*$ and $ \left\langle F({x^*}),  {x^*} \right\rangle=0$, where ${\cal K}^*$ denotes the dual of  ${\cal K}$. For a comprehensive study of theory and applications of variational inequality, see  \cite{FacchineiPang2003-II, FacchineiPang2003-I}.

The extragradient method, one of the most popular algorithms to solve the VIP($F,{\cal C}$),  was proposed in \cite{Korpelevivc76} in the 1970s and continues to attract the interest of variational inequality experts, see \cite{doi:10.1080/02331934.2010.539689,doi:10.1137/14097238X,doi:10.1080/02331934.2017.1377199,millanbello}, and the references therein.   This method is attractive due to the fact that it requires only two operator evaluations for each iteration, making it numerically stable and hence potentially suited for addressing large-scale problems.  Apart from the projection required in its definition, which is responsible for nearly all of the computational demands when the constraint set projection is challenging to compute, it is a relatively simple method. Additionally, the method converges with mild assumptions. All these features motivate the study of it, resulting in many different versions of the method throughout the years culminating in a large body of literature on the subject, including  \cite{FacchineiPang2003-I,FacchineiPang2003-II, BelloMillan,wangxiu,buradutta,doi:10.1137/S0363012998339745}, and the references therein.

Another subject that has motivated the development of numerical methods for addressing constrained problems is dealing with the computation of the projection, which is the step that accounts for practically all of the computing demands of methods that utilize projections, like the Extragradient method. In general, computing the projection requires solving a quadratic problem constrained to the feasible set at each iteration, which can significantly raise the cost per iteration if the number of unknowns is large. In light of this, it may not be reasonable to do exact projections when the iterates of the method are distant from the solution of the problem under consideration.  Over the years, various inexact procedures, which become increasingly accurate as the solution of the problem under consideration is approached, have been proposed in an effort to reduce the computational cost required for projections, thus leading to more effective projection-based methods; see for example \cite{BirginMartinezRaydan2003, Bonettini2016, Golbabaee_Davies2018, Gonccalves2020, SalzoVilla2012, VillaSalzo2013, Rasch2020,MillanFerreiraUgon, ReinierOrizonLeandro2019}.

The purpose of this paper is to present two variants of the extragradient method that utilize feasible inexact projections. In the proposed variants of the extragradient method, we will adopt a version of the scheme proposed in \cite[Example 1]{VillaSalzo2013}, where the inexact projection over the feasible set is computed with an appropriate relative error tolerance. Firstly, we present a variation of the extragradient method with constant stepsize and show that it preserves the same convergence result as the classic method, see \cite{FacchineiPang2003-II, Korpelevivc76}. We show that if $F$ is a pseudo monotone operator on $ {\cal C}$ with respect to $ {\cal C}^*$ and Lipschitz continuous, the sequence generated converges to a solution of the VIP($F,{\cal C}$).  It is important to note in this version that the Lipschitz constant is required to compute the stepsize.  Considering that the Lipschitz constant is not accessible or difficult to compute in almost every application, we propose and analyse a feasible inexact projection version of the extragradient method using an Armiljo-type line search.  It is worth noting that, like the classical extragradient method, the method does just two projections into the feasible set in each iteration.  The full convergence of the sequence to a solution is shown, with $F$ being a pseudo monotone operator on ${\cal C}$ with respect to ${\cal C}^*$ and no Lipschitz continuity assumption, which is the same results as the version with exact projection, see \cite{millanbello,buramillan,RePEc,konov,solodsvaiter}.

The organization of the paper is as follows. In section \ref{sec:Preliminares}, we present some notation and basic results used throughout the paper. In  Section~\ref{Sec:InexcProj}   we will revisit the concept of feasible inexact projection onto a closed and convex set and describe some new properties of the feasible inexact projection.   Section~\ref{Sec:ExtraGradmeMethod}  describe and analyze the extragradient method with a feasible inexact projection for solving problem~\eqref{eq:mp}.  In Section~\ref{Se:ExtGradLineSerch} is introduced and analyzed an inexact variant of the extragradient method with line search for solve  VIP($F,{\cal C}$).  Finally, some concluding remarks are made in Section~\ref{Sec:Conclusions}.

\section{{Preliminaries}} \label{sec:Preliminares}
In this section, we present some preliminary results used throughout the paper.  We denote: ${\mathbb{N}}:=\{1,2,3, \ldots\}$, $\langle \cdot,\cdot \rangle$ is the usual inner product and  $\|\cdot\|$ is the Euclidean norm.  Let    ${\cal C}\subset  \mathbb{R}^n $ be closed, convex and nonempty  set, the {\it projection} is   the map    ${\PJ}_{\cal C}:   \mathbb{R}^n \to {\cal C}$   defined by
\begin{equation*}
{\PJ}_{\cal C}(v):=  \arg\min_{y \in {\cal C}}\|v-z\|.
\end{equation*} 
In the next lemma, we present some important properties of the projection mapping. 
\begin{lemma}\label{le:projeccion}
 Given a convex and closed set ${\cal C}\subset \mathbb{R}^n$ and $v\in  \mathbb{R}^n$, the following properties hold:
\begin{enumerate}
\item[(i)]$\langle v-{\PJ}_{\cal C}(v), z-{\PJ}_{\cal C}(v)\rangle\leq0$, for all $z\in {\cal C}$;
\item[(ii)] $\|{\PJ}_{\cal C}(v)-z\|^2\leq \|v-z\|^2- \|{\PJ}_{\cal C}(v)-v\|^2$, for all $z\in {\cal C}$.
\end{enumerate}
\end{lemma}
\begin{proof}
The item (i)  is proved in  \cite[Theorem 3.14]{BauschkeCombettes20011}. For item (ii),  combine  $ \|v-z\|^2=\|{\PJ}_{\cal C}(v)-v\|^2+\|{\PJ}_{\cal C}(v)-z\|^2-2\langle {\PJ}_{\cal C}(v)-v, {\PJ}_{\cal C}(v)-z\rangle$ with  item (i). 
\end{proof}
For the formula in the next proposition see, for example,  \cite[Example 3.21]{BauschkeCombettes20011}.
\begin{proposition}\label{prop:0}
Let  $a, v \in  \mathbb{R}^n$ and $H=\{x\in  \mathbb{R}^n:~\langle v, x- a\rangle \leq 0 \}$. If  ${\bar x} \notin H$,  then
$$
{\PJ}_{H}({\bar x})= {\bar x}-\frac{1}{\Vert v\Vert ^{2}}\langle v, {\bar x}- a \rangle v.
$$
\end {proposition}
Let  $F: {\mathbb R}^n\to {\mathbb R}^n$ be an operator, ${\cal C}\subset {\mathbb R}^n$  be a nonempty and closed convex set.  The operator  $F$ is said to be {\it pseudo-monotone on  ${\cal C}$ with respect to the solution set  ${\cal C}^*$} of  problem~\eqref{eq:mp} if the set  ${\cal C}^*$ is  nonempty and,  for every  $x^*\in {\cal C}^*$,  there holds:
\begin{equation*} 
\left\langle F(x),  x-{x^*} \right\rangle\geq 0, \qquad \forall x\in {\cal C}.
\end{equation*}
It is worth noting that the concept described above represents a relaxation of the pseudo-monotonicity of $F$ on ${\cal C}$.  The concept of {\it pseudo-monotonicity of $F$ on ${\cal C}$} is defined as follows: for all $x, y \in {\cal C}$, one hast that  $\left\langle F(y), x - y \right\rangle\geq 0$ implies  $\left\langle F(x), x - y \right\rangle \geq 0$.

\begin{definition}\label{def:fejer}
Let $S$ be a nonempty subset of $\mathbb{R}^n$. A sequence $( v_k)_{k\in\mathbb{N}}\subset \mathbb{R}^n$ is said to be quasi-Fej\'er convergent to $S$, if and only if, for all $v\in S$ there exists ${\bar k}\ge 0$ and a summable sequence $(\epsilon_k)_{k\in\mathbb{N}}$, such that $\|v_{k+1}-v\|^2 \le \| v_k - v\|^2+\epsilon_k$ for all $k\ge {\bar k}$. 
\end{definition}
In the following lemma, we state the main properties of quasi-Fej\'er sequences that we will need; a comprehensive study on this topic can be found in \cite{Combettes2001}.
\begin{lemma}\label{le:fejer}
Let $S$ be a nonempty subset of $\mathbb{R}^n$ and   $(v_k)_{k\in\mathbb{N}}$  be  a quasi-Fej\'er sequence convergent to $S$.  Then, the following conditions hold:
\begin{enumerate}
\item[(i)] the sequence $(v_k)_{k\in\mathbb{N}}$ is bounded;
\item[(ii)] if a cluster point ${\bar v}$ of $(v_k)_{k\in\mathbb{N}}$  belongs to $S$, then  $(v_k)_{k\in\mathbb{N}}$ converges to ${\bar v}$.
\end{enumerate}
\end{lemma}
\section{Feasible inexact projection} \label{Sec:InexcProj}
In this section, we will revisit the concept of feasible inexact projection onto a closed and convex set. This concept has already been utilized in \cite{AdOriLea2020,OriFabaGil2018,ReinierOrizonLeandro2019, MillanFerreiraUgon}. We also describe some new properties of the feasible inexact projection, which is employed throughout the work. The definition of feasible inexact projection is as follows.

\begin{definition} \label{def:InexactProj}
	Let ${\cal C}\subset {\mathbb R}^n$ be a closed convex set and $\gamma \in {\mathbb R}_{+}$ is a given error tolerance. 
 The {\it feasible inexact projection mapping} relative to $u \in {\cal C}$ with error tolerance  ${\gamma}$, denoted by ${\cal P}^{\gamma}_{\cal C}(u, \cdot): {\mathbb R}^n \rightrightarrows {\cal C}$, is the set-valued mapping defined as follows
\begin{equation} \label{eq:Projw} 
	{\cal P}^{\gamma}_{\cal C}(u, v) := \left\{w\in {\cal C}:~\big\langle v-w, y-w \big\rangle \leq  \gamma \|w-u\|^2,~\forall y \in {\cal C} \right\}.
\end{equation}
Each point $w\in {\cal P}^{\gamma}_{\cal C}( u, v)$ is called a {\it feasible inexact projection of $v$ onto ${\cal C}$ relative to $u$  with error tolerance $\varphi_{\gamma}$}.
\end{definition}

The feasible inexact projection generalizes the concept of usual projection. In the following,  we present some remarks about this concept. 
\begin{remark}\label{rem: welldef}
Let $\gamma \in {\mathbb R}_{+}$ be  error tolerance, ${\cal C}\subset {\mathbb R}^n$, $u\in {\cal C}$ and ${\gamma}$ be as in Definition~\ref{def:InexactProj}.  For all $v\in {\mathbb R}^n$, it follows from \eqref{eq:Projw} that ${\cal P}^0_{\cal C}(u, v)$ is the exact projection of $v$ onto ${\cal C}$; see \cite[Proposition~2.1.3, p. 201]{Bertsekas1999}. Moreover, ${\cal P}^0_{\cal C}( u, v) \subset {\cal P}^{\gamma}_{\cal C}( u, v)$  which implies that  ${\cal P}^{\gamma}_{\cal C}( u, v)\neq \varnothing$, for all $u\in {\cal C}$ and $v\in {\mathbb R}^n$. In general, if $\gamma\leq {\bar \gamma}$ then ${\cal P}^{\gamma}_{\cal C}( u, v) \subset {\cal P}^{\bar \gamma}_{\cal C}( u, v)$.
\end{remark}

Below we present a particular counterpart of the firm non-expansiveness of the projection operator to a feasible inexact projection operator, its proof follows the same idea of  \cite{ReinierOrizonLeandro2019}.
\begin{proposition} \label{pr:snonexp}
  Let $v\in {\mathbb R}^n$ and   $ \gamma\geq 0$. If  ${w}\in {\PJ}^{\gamma}_C({u},v)$ and  ${\bar w}= {\PJ}_C({\bar v})$, then  
 \begin{equation*}
 \|{w}-{\bar w}\|^2\leq \|{v}-{\bar v}\|^2- \|({v}-{\bar v})-({w}-{\bar w})\|^2  +2\gamma\|{w}-{u}\|^2.
\end{equation*} 
\end{proposition}
\begin{proof}
Since   ${w}\in {\PJ}^{\gamma}_C({u},v)$   and  ${\bar w}= {\PJ}_C({\bar v})$,  it follows from \eqref{eq:Projw} and Lemma~\ref{le:projeccion} that
$$
\big\langle v-w, {\bar w}-w \big\rangle \leq  \gamma \|w-u\|^2, \qquad \big\langle {\bar v}-{\bar w}, w-{\bar w} \big\rangle \leq  0
$$
By adding the last two inequalities,   some algebraic manipulations yield 
\begin{equation*}
-\big\langle {\bar v}-{v}, {\bar w}-w \big\rangle +  \|{w}-{\bar w}\|^2 \leq   \gamma \|w-u\|^2.
\end{equation*}
Since  $\|({\bar v}-{v})- ({\bar w}-w)\|^2=\|{\bar v}-{v}\|^2-2\big\langle {\bar v}-{v}, {\bar w}-w \big\rangle  +\|{\bar w}-{w}\|^2$, the desired inequality follows by combination with the last inequality.
\end{proof}

The forthcoming lemma delineates essential properties of the feasible inexact projection operator, specifically crafted for examining the extragradient method. These properties will play a crucial role in deriving features associated with the progress of this  algorithm, as well as in verifying optimality.

\begin{lemma} \label{Le:ProjProperty}
	Let  $F: {\mathbb R}^n\to {\mathbb R}^n$ be an   operator, $ {\cal C}\subset {\mathbb R}^n$   be a nonempty, closed, and convex set, $x \in {\cal C}$ and  $0\leq \gamma  < 1$. Take  $z\in  {\mathbb R}^n$ and any inexact projection 
	$$
	w(\alpha)  \in {\cal P}^{\gamma}_{\cal C}(x, x-\alpha F(z)), \qquad \alpha  \in(0, +\infty).
	$$
	Then, there hold:
	\begin{enumerate}
		\item[(i)] $\big\langle F(z), w(\alpha) - x\big\rangle \leq \dfrac{\gamma-1}{\alpha}\|w(\alpha)-x\|^2$;
		\item[(ii)] $\|w(\alpha)-x\|\leq\dfrac{\alpha}{1-\gamma}\|F(z)\|$. 
	\end{enumerate}
\end{lemma}
\begin{proof}
	Since $w(\alpha)  \in {\cal P}^{\gamma}_{\cal C}(x, x-\alpha F(z))$ we obtain $\big\langle x-\alpha F(z)-w(\alpha), x-w(\alpha) \big\rangle \leq \gamma \|w(\alpha)-x\|^2$,  which after some algebraic manipulations yields
$$
 \|w(\alpha)-x\|^2-\alpha \big\langle  F(z), x-w(\alpha) \big\rangle \leq \gamma \|w(\alpha)-x\|^2.
$$
Thus, item $(i)$ follows from the last inequality.  We proceed to prove the tem $(ii)$. For that, first note that  the item $(i)$ is equivalent to 
\begin{equation} \label{eq:aitm1}
 0\leq \dfrac{1-{\gamma}}{\alpha}\|w(\alpha)-x\|^2\leq  \big\langle F(z),   x-w(\alpha)\big\rangle.
\end{equation} 
 If $w(\alpha)=x$, then the inequality holds trivially.  Assume that $w(\alpha)\neq x$. Thus, the inequality  in item (ii) follows by combining the inequality \eqref{eq:aitm1} with 
 $\langle F(z),   x-w(\alpha)\rangle\leq \|F(z)\|\|x-w(\alpha)\|$.
\end{proof}
The subsequent corollary, derived from the preceding lemma, plays a significant role in verifying the optimality of points generated by the extragradient method.
\begin{corollary} \label{cr:equival}
  The following statements are equivalent:
  \begin{enumerate}
  \item[(i)] $x$ is a solution of the VIP(F,${\cal C}$);
  \item[(ii)] $x\in {\cal P}^{\gamma}_{\cal C}(x,x-\alpha F(x))$, for all $\alpha  \in(0, +\infty)$;
  \item[(iii)]  there exists ${\bar \alpha}>0$ such that $\langle F(x),w({\bar \alpha})-x\rangle\geq 0$ for $w({\bar \alpha})\in {\cal P}^{\gamma}_{\cal C}(x,x-{\bar \alpha} F(x))$.
  \end{enumerate}
\end{corollary}
\begin{proof}
Proof of equivalence between item $(i)$ and item $(ii)$: We first assume that item $(i)$ holds, i.e., $x$ is a solution for problem \eqref{eq:mp}. In this case, by taking $w(\alpha) \in {\cal P}^{\gamma}_{\cal C}(x, x-\alpha F(x))$, we find that $w(\alpha) \in {\cal C}$. Consequently, we have $\big\langle F(x), w(\alpha)-x \big\rangle \geq 0$. Considering that $\alpha > 0$ and $0 \leq \gamma < 1$, the last inequality, along with item $(i)$ of Lemma~\ref{Le:ProjProperty} for $z=x$, implies that $w(\alpha) = x$. Hence, $x \in {\cal P}^{\gamma}_{\cal C}(x, x-\alpha F(x))$, and item $(ii)$ also holds. Reciprocally, assuming that item $(ii)$ holds, if $x \in {\cal P}^{\gamma}_{\cal C}(x, x-\alpha F(x))$, then applying \eqref{eq:Projw} with $w = x$, $v = x-\alpha F(x)$, and $u = x$ yields $\big\langle x-\alpha F(x)-x, y-x \big\rangle \leq 0$, for all $y \in {\cal C}$. Given that $\alpha > 0$, the last inequality is equivalent to $\big\langle F(x), y-x \big\rangle \geq 0$, for all $y \in {\cal C}$. Thus, $x$ is a solution for problem \eqref{eq:mp}, and item $(i)$ holds as well.

Proof of equivalence between item $(ii)$ and item $(iii)$: Let us  assume that item $(ii)$ holds. Thus, item $(i)$ also holds, and $x$ is a solution for problem \eqref{eq:mp}, which implies that $\left\langle F(x), y-x \right\rangle \geq 0$ for all $y\in {\cal C}$. Considering that for any $w({\bar \alpha})\in {\cal P}^{\gamma}_{\cal C}(x,x-{\bar \alpha} F(x))$, we have $w({\bar \alpha})\in {\cal C}$, it follows that $\langle F(x),w({\bar \alpha})-x\rangle\geq 0$, and item $(iii)$ holds. Conversely, we assume, for contradiction, that item $(ii)$ does not hold. Therefore, $x \notin {\cal P}^{\gamma}_{\cal C}(x, x-\alpha F(x))$, and considering that $w(\alpha) \in {\cal P}^{\gamma}_{\cal C}(x, x-\alpha F(x))$, we conclude that $x \neq w(\alpha)$. As a result, because $\alpha > 0$ and $0 < \gamma \leq \bar{\gamma}$, it follows from item $(i)$ of Lemma~\ref{Le:ProjProperty} that $\big\langle F(x), w(\alpha) - x \big\rangle < 0$, for all $\alpha \in (0, +\infty)$. Thus, item~$(iii)$ does not hold, which leads to a contradiction. Therefore, $(iii)$ implies $(ii)$.
\end{proof}

This section concludes with the presentation of a consequence derived from Lemma~\ref{Le:ProjProperty}, providing a valuable tool for verifying the progress of the extragradient method.
\begin{corollary} \label{cr:ProjPropertyii}
Let  $F: {\mathbb R}^n\to {\mathbb R}^n$ be a   operator, $ {\cal C}\subset {\mathbb R}^n$   be a nonempty, closed, and convex set, $x \in {\cal C}$, $0\leq \gamma  < 1$ and $\alpha>0$.   If $y  \in {\cal P}^{\gamma}_{\cal C}(x, x-\alpha F(x))$ and  $x^{+}  \in {\cal P}^{\gamma}_{\cal C}(x, x-\alpha F(y))$, then the following inequalities   hold:
\begin{enumerate}
		\item[(i)] $\|y-x\|\leq\dfrac{\alpha}{1-\gamma}\|F(x)\|$;
		\item[(ii)] $\|x^{+}-x\|\leq\dfrac{\alpha}{1-\gamma}\|F(y)\|$.
\end{enumerate}
As a consequence, if   $F$ is   Lipschitz continuous on ${\cal C}$ with  constant $L > 0$, then it holds:
\begin{equation} \label{eq:ProjPropertyii}
\|x^{+}-x\|\leq \dfrac{\alpha(1-\gamma+\alpha L)}{(1-\gamma)^2}\|F(x)\|.
\end{equation}
\end{corollary} 
\begin{proof}
Applying   item $(ii)$ of Lemma~\ref{Le:ProjProperty} with  $w(\alpha)=y$ and $z=x$ we obtain  item $(i)$  and  with  $w(\alpha)=x^{+}$ and $z=y$ we obtain  item (ii). We proceed to prove \eqref{eq:ProjPropertyii}. For that, we first note that  due to  $F$  be Lipschitz continuous on ${\cal C}$ with  constant $L > 0$, we obtain that 
$$
\|F(y)\|\leq  \|F(y)-F(x)\|+\|F(x)\|\leq L\|y-x\|+\|F(x)\|.
$$
Thus, considering that $y  \in {\cal P}^{\gamma}_{\cal C}(x, x-\alpha F(x))$, we can apply the  item $(i)$  to obtain 
$$
\|F(y)\|\leq  \dfrac{\alpha L}{1-\gamma}\|F(x)\|+\|F(x)\|= \dfrac{1-\gamma+\alpha L}{1-\gamma}\|F(x)\|.
$$
By combing the last inequality with item $(ii)$, the desired inequality follows.
\end{proof}
\section{Extragradient inexact  method with constant step size} \label{Sec:ExtraGradmeMethod}
In this section, we describe the extragradient method with a feasible inexact projection for solving problem~\eqref{eq:mp}. It should be noted that the proposed method uses appropriate relative error criteria to compute inexact projections on the constraint set, in contrast to the extragradient method, which employs exact projections on the constraint set. The inexact version of the  classical extragradient method   is stated  as follows:
\begin{algorithm}[h]
\caption{\footnotesize \bf Extragradient  inexact  projection method-EInexPM}
\begin{footnotesize}
	\begin{algorithmic}[1] 
	\State {Take  $\alpha>0$, $0<\bar{\gamma}<1/2$ and  $(a_k)_{k\in\mathbb{N}}$  satisfying  $\sum_{k \in \mathbb{N}} a_k<+\infty$. Let $x^{1}\in {\cal C}$ and set $k=1$.}
          \State{
	 Choose an error tolerance  $\gamma_k$ such that  
				\begin{equation} \label{eq:Tolerance}
					0 \leq \gamma_k \|F(x^{k})\|^2  \leq a_k, \qquad  0 \leq \gamma_k < \bar{\gamma},
				\end{equation}
	and  compute the  following feasible inexact projections:
	\begin{align}
					y^{k} &\in {\cal P}^{\gamma_k}_{\cal C}\big(x^{k}, x^k-\alpha F(x^k) \big); \label{eq:IntStep1}\\
					x^{k+1} &\in {\cal P}^{\gamma_k}_{\cal C}\big(x^{k}, x^k-\alpha F(y^k) \big). \label{eq:IntStep2}
	\end{align}
	}
        \State{If $x^k=y^k$ or $y^k=x^{k+1}$, {\bf stop}, else set $k\gets k+1$, and go to step 2.}
	\end{algorithmic}
\end{footnotesize}
 \label{Alg:ExtraGradMethod}
\end{algorithm}

Let us examine the main features of the  EInexPM. To begin, we select a constant step size $\alpha>0$, an upper bound for error tolerances ${\bar \gamma}$ such that $0<\bar{\gamma}<1/2$ and select  an exogenous summable sequence $(a_k)_{k\in\mathbb{N}}$ to control the error tolerance.   The stopping criterion $F(x^{k}) = 0$ is then evaluated in the current iteration $x^{k}$. If this criterion is not satisfied, a non-negative error tolerance $\gamma_k$ that fulfils the requirements \eqref{eq:Tolerance} is selected. By using an inner procedure compute $y^k$ as any feasible inexact projection $x^k - \alpha F(x^k)$ onto the feasible set  ${\cal C}$ relative to $x^k$, i.e.  $y^{k} \in {\cal P}^{\gamma_k}_{\cal C}\big(x^{k}, x^k-\alpha F(x^k) \big)$. Finally,  using again an inner procedure,  the next iterate $x^{k+1}$ is computed as any feasible inexact projection of $x^k - \alpha F(y^k)$ onto the feasible set ${\cal C}$ relative to $x^k$, i.e., $x^{k+1} \in {\cal P}^{\gamma_k}_{\cal C}\big(x^{k}, x^k-\alpha F(y^k) \big)$. 
 
It is worth noting that if $\gamma_k \equiv 0$, then Remark~\ref{rem: welldef} implies that inexact projections are the exact ones. Hence, EInexPM corresponds to the classical extragradient method introduced in \cite{Korpelevivc76}.  It is important to note that $\gamma_k$  in \eqref{eq:Tolerance} can be selected as any nonnegative real number fulfilling  $0 \leq \gamma_k  \|F(x^{k})\|^2 \leq a_k$, for a prefixed sequence $(a_k)_{k\in\mathbb{N}}$. In this case, we have
\begin{equation} \label{eq:ToleranceCond}
\sum_{k \in \mathbb{N}} \big(\gamma_k\|F(x^k)\|^2\big) < +\infty.
\end{equation}
Since  $x^{1}\in {\cal C}$ and,  for all $k \in \mathbb{N}$,  $x^{k+1}$ is a feasible inexact projection onto ${\cal C}$, we conclude  $(x^k)_{k\in\mathbb{N}}\subset {\cal C}$. As a consequence of  ${\cal C}$ being a closed set, any cluster point of $(x^k)_{k\in\mathbb{N}}$, if any exists, belongs to ${\cal C}$. 

Next we present two examples of sequences $(a_k)_{k\in\mathbb{N}}$ satisfying  $\sum_{k \in \mathbb{N}} a_k<+\infty$.
\begin{example}
	Sequences  $(a_k)_{k\in\mathbb{N}}$ satisfying  $\sum_{k \in \mathbb{N}} a_k<+\infty$ are obtained  by taking $a_{k}:=b_{k-1}-b_{k}$ and  $\bar{b}>0$ satisfying one the following conditions:  (i) $b_0=2\bar{b}$, $b_k=\bar{b}/k$, for all $k=1, 2, \ldots$; (ii)$b_0=2\bar{b}$, $b_k=\bar{b}/\ln(k+1)$, for all $k=1, 2, \ldots$.
\end{example}
The convergence analysis of the sequence  $(x^k)_{k\in\mathbb{N}}$ produced by EInexPM will be discussed in the following sections.
\subsection{Convergence analysis}
We will show in this section that the sequence   $(x^k)_{k\in\mathbb{N}}$ generated by EInexPM converges to a solution of VIP(F, ${\cal C}$) for Lipschitz continuous pseudo-monotone  operator $F$ with Lipschitz constant $L\geq 0$.  To state our first result, let us recall that the solution set of  VIP(F, ${\cal C}$) is denoted by ${\cal C}^*$ and define the following constants:
\begin{equation} \label{eq:constetanu}
 {\bar \eta}:=1- \alpha^2 L^2-2{\bar \gamma}, \qquad \qquad  {\bar \nu}:=\dfrac{\alpha^2(1-{\bar \gamma}+\alpha L)^2}{(1-{\bar \gamma})^4}.
 \end{equation}
 
The convergence of the sequence $(x^k)_{k \in \mathbb{N}}$, generated by Algorithm~\ref{Alg:ExtraGradMethod}, relies on the following lemma. As we will demonstrate, under the condition \eqref{eq:Tolerance}, this lemma implies that $(x^k)_{k \in \mathbb{N}}$ is quasi-Fej'er convergent to the set ${\cal C}^*$.
 
\begin{lemma}\label{le:FejerProperty}
Let ${\cal C}\subset {\mathbb R}^n$  be  a nonempty and closed  convex set and $(x^k)_{k\in{\mathbb N}}$ the sequence generated by Algorithm~\ref{Alg:ExtraGradMethod}. Assume that   $F: {\mathbb R}^n\to {\mathbb R}^n$ is a   pseudo monotone operator on  ${\cal C}$ with respect to  ${\cal C}^*$ and  Lipschitz continuous on ${\cal C}$ with constant $L > 0$. Then, for any $x^*\in {\cal C}^*$, there holds:
\begin{equation*}
\|x^{k+1}-x^*\|^2\leq \|x^k-x^*\|^2-  {\bar \eta} \|x^{k}-y^{k}\|^2 + {\bar \nu} \gamma_k\|F(x^k)\|^2, \qquad k=1, 2, \ldots.
\end{equation*} 
\end{lemma} 
\begin{proof}
First note that 
\begin{multline*}
\big\langle x^{k}-\alpha F(y^{k})-y^{k}, x^{k+1}-y^{k}\big\rangle=\big\langle x^{k}-\alpha F(x^{k})-y^{k}, x^{k+1}-y^{k}\big\rangle \\+ \alpha\big\langle F(x^{k})- F(y^{k}), x^{k+1}-y^{k}\big\rangle.
\end{multline*}
Since \eqref{eq:IntStep1} implies that $y^{k} \in {\cal P}^{\gamma_k}_{\cal C}\big(x^{k}, x^k - \alpha F(x^k) \big)$,  we conclude that 
\begin{equation} \label{eq:faip}
\big\langle x^{k}-\alpha F(y^{k})-y^{k}, x^{k+1}-y^{k}\big\rangle\leq \gamma_k \|y^{k}-x^{k}\|^2+ \alpha\big\langle  F(x^{k})- F(y^{k}),x^{k+1}-y^{k}\big\rangle.
\end{equation}
For simplicity set $z^k:=x^k-\alpha F(y^k)$. As nothing more than a result of some algebraic manipulations, we arrive to the conclusion that
\begin{align*}
\|x^{k+1}-x^*\|^2&=\|x^{k+1}-z^{k}\|^2+ \|z^{k}-x^*\|^2-2 \langle x^{k+1}-z^{k},  x^*-z^{k} \rangle \\
                         &=-\|x^{k+1}-z^{k}\|^2+ \|z^{k}-x^*\|^2+2 \langle z^{k}-x^{k+1},  x^*- x^{k+1} \rangle.
\end{align*}
Using that $x^{k+1} \in {\cal P}^{\gamma_k}_{\cal C}\big(x^{k}, z^k \big)$ we conclude  that $ \langle z^{k}-x^{k+1},  x^*- x^{k+1} \rangle \leq \gamma_k \|x^{k+1}-x^{k}\|^2$, which combined with the previous equality give us 
\begin{equation}\label{eq:poee}
\|x^{k+1}-x^*\|^2\leq \|z^{k}-x^*\|^2-\|x^{k+1}-z^{k}\|^2+2\gamma_k \|x^{k+1}-x^{k}\|^2
\end{equation}
Taking into account that $z^k=x^k-\alpha F(y^k)$, some calculations show that
\begin{align*}
\|z^{k}-x^*\|^2-\|x^{k+1}-z^{k}\|^2&=\|x^k-x^*-\alpha F(y^k)\|^2-\|x^k-x^{k+1}-\alpha F(y^k)\|^2\notag\\
                                                   &= \|x^k-x^*\|^2-\|x^k-x^{k+1}\|^2+2\alpha\big\langle F(y^k), x^*-x^{k+1}\big\rangle.
\end{align*}
On the other hand, considering that $x^*\in {\cal C}^*$, $y^{k}\in {\cal C}$ and $F$ is  pseudo monotone  operator  on  ${\cal C}$ with respect to  ${\cal C}^*$ we have   $\big\langle F(y^{k}),  y^{k}-{x^*} \big\rangle\geq 0$. Thus, we conclude that 
\begin{equation*}
\big\langle F(y^{k}),  {x^*}-x^{k+1} \big\rangle\leq \big\langle F(y^{k}),  y^{k}-x^{k+1}\big\rangle.
\end{equation*} 
The last inequality, when combined with the previous equality, implies that
\begin{equation*}
\|z^{k}-x^*\|^2-\|x^{k+1}-z^{k}\|^2\leq  \|x^k-x^*\|^2-\|x^k-x^{k+1}\|^2+2\alpha \big\langle F(y^{k}),  y^{k}-x^{k+1} \big\rangle.
\end{equation*}
The previous inequality is now combined with \eqref{eq:poee} to provide the following inequality
\begin{equation*}
\|x^{k+1}-x^*\|^2\leq \|x^k-x^*\|^2-\|x^k-x^{k+1}\|^2+\gamma_k \|x^{k+1}-x^{k}\|^2 +2\alpha \big\langle F(y^{k}),  y^{k}-x^{k+1} \big\rangle.
\end{equation*}
Since  $\|x^{k}-x^{k+1}\|^2=\|x^{k}-y^{k}\|^2 + \|y^{k}-x^{k+1}\|^2+2\langle x^{k}-y^{k}, y^{k}-x^{k+1}\rangle$, the last inequality is equivalent to 
\begin{multline*}
\|x^{k+1}-x^*\|^2\leq \|x^k-x^*\|^2-\|x^{k}-y^{k}\|^2 - \|y^{k}-x^{k+1}\|^2+\gamma_k\|x^{k+1}-x^{k}\|^2 +\\2 \big\langle x^{k} -\alpha F(y^{k})-y^{k},  x^{k+1}-y^{k}\big\rangle.
\end{multline*}
 The last inequality together with  \eqref{eq:faip} yield
\begin{multline*}
\|x^{k+1}-x^*\|^2\leq \|x^k-x^*\|^2-\|x^{k}-y^{k}\|^2 - \|y^{k}-x^{k+1}\|^2+\gamma_k\|x^{k+1}-x^{k}\|^2 \\+2\gamma_k\|y^{k}-x^{k}\|^2+ 2\alpha\big\langle  F(x^{k})- F(y^{k}),x^{k+1}-y^{k}\big\rangle.
\end{multline*}
Considering that $F$ is  Lipschitz continuous on ${\cal C}$ with  constant $L > 0$, we have 
\begin{equation*}
\langle  F(x^{k})- F(y^{k}),x^{k+1}-y^{k}\rangle\leq L \|x^{k}-y^{k}\|\|x^{k+1}-y^{k}\|, 
\end{equation*}
which combined with the last inequality  yields 
\begin{multline*}
\|x^{k+1}-x^*\|^2\leq \|x^k-x^*\|^2-\|x^{k}-y^{k}\|^2 - \|y^{k}-x^{k+1}\|^2+\gamma_k\|x^{k+1}-x^{k}\|^2 \\+2\gamma_k\|y^{k}-x^{k}\|^2+ 2\alpha L \|x^{k}-y^{k}\|\|x^{k+1}-y^{k}\|.
\end{multline*}
or equivalently, 
 \begin{multline*}
\|x^{k+1}-x^*\|^2\leq \|x^k-x^*\|^2- (1- \alpha^2 L^2-2\gamma_k) \|x^{k}-y^{k}\|^2 \\+\gamma_k\|x^{k+1}-x^{k}\|^2 -(\alpha L \|x^{k}-y^{k}\|-\|x^{k+1}-y^{k}\|)^2.
\end{multline*}
Hence, we have 
$$
\|x^{k+1}-x^*\|^2\leq \|x^k-x^*\|^2- (1- \alpha^2 L^2-2\gamma_k) \|x^{k}-y^{k}\|^2 +\gamma_k\|x^{k+1}-x^{k}\|^2. 
$$
Thus, taking into account that  $x^{k+1} \in {\cal P}^{\gamma_k}_{\cal C}\big(x^{k}, x^k-\alpha F(y^k) \big)$, by applying  Corollary~\ref{cr:ProjPropertyii} with  $x^{+}=x^{k+1}$, $x=x^{k}$ and $\gamma=\gamma_k$ we obtain that 
$$
\|x^{k+1}-x^*\|^2\leq \|x^k-x^*\|^2- (1- \alpha^2 L^2-2\gamma_k) \|x^{k}-y^{k}\|^2 + \dfrac{\alpha^2(1-\gamma_k+\alpha L)^2}{(1-\gamma_k)^4} \gamma_k\|F(x^k)\|^2. 
$$
Therefore, using \eqref{eq:constetanu} and considering  that $0 \leq \gamma_k < \bar{\gamma}$,  the desired  inequality follows.
\end{proof} 
The following theorem represents the central outcome concerning Algorithm~\ref{Alg:ExtraGradMethod}. It establishes the convergence of the sequence towards a solution of VIP(F,${\cal C}$).
\begin{theorem} \label{th:ConvTheoCEInexPM}
Let ${\cal C}\subset {\mathbb R}^n$  be  a nonempty and closed  convex set and $(x^k)_{k\in{\mathbb N}}$ the sequence generated by Algorithm~\ref{Alg:ExtraGradMethod}. Assume that   $F: {\mathbb R}^n\to {\mathbb R}^n$ is  a   pseudo monotone operator  on  ${\cal C}$ with respect to  ${\cal C}^*$ and  Lipschitz continuous on ${\cal C}$ with  constant $L > 0$. If 
\begin{equation} \label{eq:ccs}
0<\alpha<\frac{\sqrt{1-2{\bar \gamma}}}{L}, 
\end{equation} 
then  the sequence $(x^k)_{k\in{\mathbb N}}$ converges to a solution of the VIP(F,${\cal C}$).
\end{theorem} 
\begin{proof}
Let $x^*\in {\cal C}^*$ be a arbitrary solution of the VIP(F,${\cal C}$).  The condition \eqref{eq:ccs}  and $0<\bar{\gamma}<1/2$ imply that  ${\bar \eta} >0$ and ${\bar \nu} >0$. Thus,  it follows from Lemma~\ref{le:FejerProperty} that 
\begin{equation*} 
\|x^{k+1}-x^*\|^2\leq \|x^k-x^*\|^2+ {\bar \nu} \gamma_k\|F(x^k)\|^2, \qquad k=1, 2, \ldots.
\end{equation*} 
The last inequality together with \eqref{eq:ToleranceCond} implies that $(x^k)_{k\in{\mathbb N}}$  is  quasi-Fej\'er convergent to  ${\cal C}^*$. Considering that ${\cal C}^*$ is nonempty,  the item (i) of Lemma~\ref{le:fejer} implies  that $(x^k)_{k\in{\mathbb N}}$  is bounded. Let ${\bar x}$ be a  cluster point of  $(x_k)_{k\in\mathbb{N}}$ and   $(x_{k_j})_{j\in\mathbb{N}}$ a subsequence of $(x_k)_{k\in\mathbb{N}}$ such that $\lim_{j\to +\infty}x_{k_j}={\bar x}$.  To continue the proof, keep in mind that Lemma~\ref{le:FejerProperty} also implies that
\begin{equation*} 
{\bar \eta} \|x^{k}-y^{k}\|^2\leq \|x^k-x^*\|^2-\|x^{k+1}-x^*\|^2+ {\bar \nu} \gamma_k\|F(x^k)\|^2, \qquad k=1, 2, \ldots.
\end{equation*} 
By adding both sides of the previous inequality and using \eqref{eq:ToleranceCond}, we arrive at the conclusion that
\begin{equation*} 
{\bar \eta} \sum_{k=0}^{+\infty}\|x^{k}-y^{k}\|^2\leq \|x^{1}-x^*\|^2+ {\bar \nu}  \sum_{k=0}^{+\infty}\gamma_k\|F(x^k)\|^2<+\infty.
\end{equation*}
Hence, we have $\lim_{k\to +\infty}\|x^{k}-y^{k}\|=0$. Thus, taking into account that $\lim_{j\to +\infty}x_{k_j}={\bar x}$,  we conclude that $\lim_{j\to +\infty}y_{k_j}={\bar x}$. It follows from \eqref{eq:IntStep1}  that 
$$
y^{k_j} \in {\cal P}^{\gamma_{k_j}}_{\cal C}\big(x^{k_j}, x^{k_j}-\alpha F(x^{k_j}) \big).
$$
Considering that $ \gamma_{k_j} < \bar{\gamma}$, the last inclusion and  Definition~\ref{def:InexactProj} imply that 
\begin{equation*} 
\big\langle x^{k_j}-\alpha F(x^{k_j})-y^{k_j}, y-y^{k_j} \big\rangle \leq  \gamma \|y^{k_j}-x^{k_j}\|^2,\qquad \qquad ~\forall y \in {\cal C}.
\end{equation*}
Since $\lim_{j\to +\infty}x_{k_j}={\bar x}$ and $\lim_{j\to +\infty}y_{k_j}={\bar x}$, taking  the limit in the previous inequality as $j$ tending to infinity yields
\begin{equation*}
\big\langle {\bar x}-\alpha F({\bar x})-{\bar x}, y-{\bar x}\big\rangle \leq {\bar \gamma}\|{\bar x}-{\bar x}\| {},\qquad \qquad ~\forall y \in {\cal C}.
\end{equation*}
which, by using that $\alpha>0$,  is equivalent to $\big\langle F({\bar x}), y-{\bar x}\big\rangle \geq  0$, for all $y \in {\cal C}$. Hence, ${\bar x}\in {\cal C}^*$.  Given that ${\bar x}$ is a cluster point of $(x_k)_{k\in\mathbb{N}}$, item (ii) of Lemma~\ref{le:fejer} implies that $\lim_{k\to +\infty}x_{k}={\bar x}$, and the proof is complete.
\end{proof} 
\section{Extragradient inexact method with line search} \label{Se:ExtGradLineSerch}
In this section, we introduce  an inexact variant of the extragradient method for VIP($F,{\cal C}$) with  $F$ pseudo monotone  operator   on  ${\cal C}$ with respect to the solution set  ${\cal C}^*$ of  problem~\eqref{eq:mp},  see for example \cite{IusemSvaiter1997,millanbello,konov}. Instead of a fixed step size, the method presented finds a suitable step size in each iteration by performing an Armijo type line search. It is worth noting that, like the classical extragradient method, the method does just two projections into the feasible set in each iteration. A full convergence analysis is provided, with no Lipschitz continuity assumption of the operator defining the variational inequality problem.

The inexact  version of the  proposed version of the  extragradient method   is stated  as follows:

\begin{algorithm}[h]
\begin{footnotesize}
	\begin{algorithmic}[1]
	\State {Take $0< {\hat \beta}\leq{\bar \beta}$ and  $(\beta_k)_{k\in\mathbb{N}}$  satisfying  $0< {\hat \beta}\leq\beta_k\leq{\bar \beta}$. Take also  $\sigma$, $\rho$ and  $\alpha \in (0,1)$, and 
	\begin{equation} \label{eq:bargamma}
	0<\bar{\gamma}<\min\big\{1-\rho, 2-\sqrt{3}\big\}.
	\end{equation}
	Let $x^1\in {\cal C}$ and set $k=1$.}
          \vspace{0.2cm}
          \State{
	Choose a error tolerance  $\gamma_k$ such that  $0 \leq \gamma_k < \bar{\gamma}$ and compute the  following feasible inexact projection:
				\begin{equation} \label{eq:ipyk}
					y^{k} \in {\cal P}^{\gamma_k}_{\cal C}\big(x^{k}, x^k-\beta_k F(x^k) \big).
				\end{equation}
	}
	 \State{ If $y^{k}=x^{k}$, then {\bf stop}; otherwise, compute
	        \begin{equation} \label{eq:algjk}
i_{k}:=\min \Big\{ i\in {\mathbb N}:~\big{\langle} F {\big(}x^k+\sigma \alpha^{i} (y^k-x^k) {\big )},y^{k}-x^{k} \big{\rangle} \leq \rho  \big{\langle} F(x^k), y^{k}-x^{k}\big{\rangle} \Big \}, 
                  \end{equation}
                  and set 
                  \begin{equation} \label{eq:ipzk}
					z^{k}:=x^{k}+\sigma \alpha^{i_k}(y^{k}-x^k) .
		 \end{equation}
		}
	\State{Compute the next iteration as a feasible inexact projection:
 \begin{equation} \label{eq:etakfs}
x^{k+1} \in {\cal P}^{\gamma_k}_{\cal C}\big(x^{k}, x^k-\lambda_k F(z^k) \big), \qquad \quad \lambda_k:=-\frac{1}{\|F(z^{k})\|^2}\big{\langle} F( z^k) ,z^{k}-x^k \big{\rangle}.
\end{equation}      
	        }
 \vspace{0.2cm}
 
        \State{ Update   $k \leftarrow k+1$ and go to Step~2.}
	\end{algorithmic}
\end{footnotesize}	
\caption{{\footnotesize \bf Extragradient inexact projection method with line search-EInexPMLS}}
\label{Alg:ExtraGradMethodLS}
\end{algorithm}

Let us go through the main features of the  EInexPMLS. First, we must select some parameters that will control the behaviour of the algorithm and will be essential in your convergence analysis. The most important of these parameters is the upper bound ${\bar \gamma}$ for  error tolerance ${\gamma_k}$, which is connected to the line search parameter $\rho$.  In step 2  of the algorithm,   by using an inner procedure compute $y^k$ as any feasible inexact projection of $x^k - \beta_k F(x^k)$ onto the feasible set  ${\cal C}$ relative to $x^k$, i.e. \eqref{eq:ipyk}. Then,  in step 3, the conceptual stopping criterion $y^{k}=x^{k}$ is then evaluated in the current iteration. If this stopping criterion  is not satisfied, a line search in the segment between the points $x^k$ and $y^k$ is done in order to decrease the mapping 
$$
(0, 1) \ni t\mapsto  \langle F(x^k +t(y^{k}-x^k)) ,y^{k}-x^k \rangle.
$$ 
In step 4, the line search resultant point $z^k$ is utilized to define the following  half space
\begin{equation} \label{eq;HypSep}
H_k:=\Big\{x\in {\mathbb R}^n:~\langle F(z^k),x-z^k \rangle\leq 0\Big\}, 
\end{equation}
whose boundary  separates the current iterate $x_k$ of the solution set ${\cal C}^*$. Then, by applying Proposition~\ref{prop:0},   is computed   the  projection of  $x^k$ onto the hyperplane  $H_k$ as follows 
 \begin{equation} \label{eq:etakc}
{\PJ}_{H_k}(x^k) =x^k-\lambda_k F(z^k), \qquad \quad \lambda_k:=-\frac{1}{\|F(z^{k})\|^2}\big{\langle} F( z^k) ,z^{k}-x^k \big{\rangle}.
\end{equation}
Finally,  using again an inner procedure,  the next iterate $x^{k+1}$ is   computed as any feasible inexact projection of  ${\PJ}_{H_k}(x^k)$ onto the feasible set ${\cal C}$ relative to $x^k$. Thus, by using \eqref{eq:etakc},  we conclude that  \eqref{eq:etakfs} is equivalently stated as follows 
 \begin{equation} \label{eq:etakns}
					x^{k+1} \in {\cal P}^{\gamma_k}_{\cal C}\big(x^{k}, {\PJ}_{H_k}(x^k)\big),
 \end{equation}
It is noteworthy that if $\gamma_k \equiv 0$, then Remark~\ref{rem: welldef} implies that inexact projection is the exact one. Consequently, EInexPMLS corresponds to  a version of the  extragradient method addressed  in \cite{IusemSvaiter1997}, see also \cite{millanbello,buramillan}.
\begin{remark}\label{rem:stop}
    The stopping criteria is well defined, i.e.,  if $x^k=y^k$, then $x^k$ is a solution of the Problem \ref{eq:mp}. In fact,  $y^{k}=x^{k}$ and \eqref{eq:ipyk} implies that 
$
x^{k} \in {\cal P}^{\gamma_k}_{\cal C}\big(x^{k}, x^k-\alpha_k F(x^k) \big).
$
Thus, it follows from Corollary~\ref{cr:equival} that $x^k\in {\cal C}^*$. 
\end{remark}
Next,  we show that the Algorithm~\ref{Alg:ExtraGradMethodLS}  is well defined, namely that there exists $i_{k}$ fulfilling \eqref{eq:algjk}.
\begin{proposition}\label{prop:wellde}
Step 3 is well-defined,  i.e.,  there exists $i_k$ satisfying  \ref{eq:algjk}.
\end{proposition}

\begin{proof}
Assume by contradiction that $\big{\langle} F {\big(}x^k+\sigma \alpha^{i} (y^k-x^k) {\big )},y^{k}-x^{k} \big{\rangle}  >  \big{\langle} F(x^k), y^{k}-x^{k}\big{\rangle}$, for all $i\in {\mathbb N}$ and $y^{k}\neq x^{k}$.  Since $F$ is continuous and $0<\alpha<1$, taking the limit in the last inequality as $i$ tends to infinity, we conclude that
$
\big{\langle} F(x^k) ,y^{k}-x^k \big{\rangle} \geq \rho  \big{\langle} F(x^k), y^{k}-x^{k}\big{\rangle}.
$
Thus, taking into account that $0<\rho<1$,  the last inequality that  
\begin{equation} \label{eq;fiwd1}
\big{\langle} F(x^k) ,y^{k}-x^k \big{\rangle} \geq 0.
\end{equation} 
Considering that $y^k\in {\cal P}^{\gamma_k}_{\cal C}(x^k,x^k-\beta_k F(x^k))$ and $x^k\in {\cal C}$, it follows from Definition~\ref{def:InexactProj} that 
$$
\langle x^k-\beta_kF(x^k)-y^k,x^k-y^k\rangle\leq \gamma_k\|y^k-x^k\|^2.
$$
Since $0< {\hat \beta}\leq\beta_k$ and $0<\bar{\gamma}<1$, we can deduce from some algebraic manipulations in the preceding inequality that
$$
0\leq \big{\langle} F(x^k) ,y^{k}-x^k \big{\rangle} \leq \frac{\gamma_k-1}{\beta_k}\|y^k-x^k\|^2<0,
$$
which is a contradiction.  Therefore, there exists $i_{k}$  satisfying  \eqref{eq:algjk}. 
\end{proof}

Let  { \it$(x_{k})_{k\in \mathbb{N}}$ be  the sequence  generated by Algorithm~\ref{Alg:ExtraGradMethodLS}}. Since $x^{1}\in {\cal C}$ and, for any $k \in \mathbb{N}$,  \eqref{eq:etakfs} implies that $x^{k+1}$ is a feasible inexact projection onto ${\cal C}$, we conclude that  $(x^k)_{k\in\mathbb{N}}\subset {\cal C}$. Consequently, due to  ${\cal C}$ be a closed set, each clustering point of $(x^k)_{k\in\mathbb{N}}$ that exists belongs to ${\cal C}$. 

\subsection{Convergence analysis}
In the previous section, we show that EInexPMLs generates a $(x^k)_{k\in\mathbb{N}}$ belonging   to the set ${\cal C}$. In this section we will show that    $(x^k)_{k\in\mathbb{N}}$ converges to a solution of VIP(F, ${\cal C}$).  To this purpose, we will begin by establishing a few initial results. First, we show that  the boundary of the half space $H_k$ defined as  in \eqref{eq;HypSep} separates   the current iterates $x_k$ of the solution set ${\cal C}^*$.
\begin{proposition}\label{prop:hk}
Let $H_k$ be defined as  in \eqref{eq;HypSep}. Then,    $x^k\in H_k$ if and only if $x^k\in \mathcal{C}^*$.
\end{proposition}
    \begin{proof}
    First, we  assume that $x^k\in H(z^k)$. Thus, taking into account   \eqref{eq:ipzk} we conclude  that 
        \[
        0\geq\langle F(z^k),x^k-z^k\rangle=\sigma\alpha^{i_k}\langle F(z^k),x^k-y^k \rangle, 
        \]
     which implies that $\langle F(z^k),y^k-x^k \rangle\geq 0$. Hence, by using \eqref{eq:algjk} and  \eqref{eq:ipzk},  we conclude that $\big{\langle} F(x^k), y^{k}-x^{k}\big{\rangle} \geq 0$. Therefore,  using \eqref{eq:ipyk} together with  Corollary~\ref{cr:equival} we conclude that $x^k\in\mathcal{C}^*$.
     
Conversely,  we  assume that $x^k\in \mathcal{C}^*$.   Since $x^k$ and $y^k$ belong to ${\cal C}$,    it follows from \eqref{eq:ipzk} that $z^k\in \mathcal{C}$ by convexity of $\mathcal{C}$, for all $k\in\mathbb{N}$. Thus, due to $x^k\in \mathcal{C}^*$, we have  $\langle F(x^k),z^k -x^k\rangle\geq 0$. Since $F$ is a pseudo-monotone operator, $\langle F({x}^k),z^k-x^k\rangle\geq 0$ implies $\langle F(z^k),z^k-x^k\rangle\geq 0$, which means that $x^k\in H_k$.
 \end{proof}
    
\begin{proposition}\label{prop:Bfb}
The following inequality holds:
 \begin{equation}\label{eq:tolim}
     \big{\langle} F(x^k),x^k-y^k\big{\rangle}\geq \frac{\max \{\rho,\sqrt{3}-1\}}{\bar{\beta}}\|y^k-x^k\|^2,  \qquad k=1, 2, \ldots.
\end{equation}
\end{proposition}
\begin{proof} 
Keeping in mind that  $y^k\in {\cal P}^{\gamma_k}_{\cal C}(x^k,x^k-\beta_k F(x^k))$ and $x^k\in {\cal C}$, from  Definition~\ref{def:InexactProj}  we have 
$$
\langle x^k-\beta_kF(x^k)-y^k,x^k-y^k\rangle\leq \gamma_k\|y^k-x^k\|^2, 
$$
which, after some algebraic manipulation, is rewritten  as follows 
\begin{equation} \label{eq:Bfbpr1}
\langle F(x^k),x^k-y^k \rangle\geq \frac{1-\gamma_k}{\beta_k}\|y^k-x^k\|^2.
\end{equation}
Considering that $0<\beta_k\leq{\bar \beta}$ and $0\leq \gamma_k<\bar{\gamma}<\min\big\{1-\rho, 2-\sqrt{3}\big\}$, we conclude that  
$$
\frac{1-\gamma_k}{\beta_k}\geq \frac{\max \{\rho,\sqrt{3}-1\}}{\bar{\beta}}.
$$
The combination of \eqref{eq:Bfbpr1} with the previous inequality yields the desired inequality.
\end{proof}
Next, we are going to establish two important inequalities to show  the convergence of Algorithm~\ref{Alg:ExtraGradMethodLS}.
 \begin{lemma}\label{le:FejerPropertyLS}
Let ${\cal C}\subset {\mathbb R}^n$  be  a nonempty and closed  convex set and $(x^k)_{k\in{\mathbb N}}$ the sequence generated by Algorithm~\ref{Alg:ExtraGradMethodLS}. Assume that   $F: {\mathbb R}^n\to {\mathbb R}^n$ is  a    pseudo monotone  operator  on  ${\cal C}$ with respect to the solution set  ${\cal C}^*$ of  problem~\eqref{eq:mp} and  $x^k\notin {\cal C}^*$, for all $k=1,2,\ldots$ Then, for any $x^*\in {\cal C}^*$, there holds:
\begin{equation} \label{eq:algf14bg}
\| x^{k+1}-x^{*} \|^{2} \leq \|x^k-x^{*}\|^{2}-\frac{1}{(1-{\bar \gamma})^2}\left({\bar \gamma}^2-4{\bar \gamma}+1\right)\lambda_k^2 \|F(z^k)\|^2, \qquad k=1, 2, \ldots.
\end{equation}
As a consequence,  $(x^k)_{k\in{\mathbb N}}$ is Fej\'er convergent to ${\cal C}^*$, i.e.,  for any $x^*\in {\cal C}^*$  there holds
\begin{equation} \label{eq:FejerPropertyLs}
\| x^{k+1}-x^{*} \| \leq \|x^k-x^{*}\|, \qquad k=1, 2, \ldots.
\end{equation} 
\end{lemma} 
\begin{proof}   
Take $x^*\in \mathcal{C}^*$.   Denotes the boundary of  $H_k$  by $L_k$, which is given by 
\begin{equation} \label{eq;hpsb}
L_k:=\{x\in {\mathbb R}^n:~\langle F(z^k),x-z^k \rangle=0\}.
\end{equation}
Since   $x^k$ and $y^k$ belong to the convex set  ${\cal C}$,   we obtain from  \eqref{eq:ipzk} that $z^k\in \mathcal{C}$, for all $k\in\mathbb{N}$. 
Thus,  due to   $F$ be  pseudo monotone on  ${\cal C}$ with respect to the solution set  ${\cal C}^*$ of  problem~\eqref{eq:mp}, we have 
\begin{equation*} 
\big\langle F(z^k),  x^k-{x^*} \big\rangle\geq 0.
\end{equation*}
The last inequality and the definition of $H_k$ in \eqref{eq;HypSep} imply  that  $x^*\in H_k$. Thus, we conclude that 
\begin{equation}\label{eq:alg12}
{\PJ}_{H_k}(x^*)=x^*.
\end{equation} 
We know as well that applying Proposition~\ref{pr:snonexp} with   $v=x^k-\lambda_k F(z^k) $, $u=x^{k}$, $\gamma=\gamma_k$, ${w}=x^{k+1}$ and  ${\bar w}={\bar v}=x^*$ we obtain that 
\begin{equation*}
\|x^{k+1}-x^{*}\|^{2} \leq \|x^k-\lambda_k F(z^k) -x^{*}\|^{2}- \|(x^k-\lambda_k F(z^k)-x^{*})-({x^{k+1}}-x^*)\|^2+2\gamma_k\|x^{k+1}-x^k \|^2,
\end{equation*}
which implies that 
\begin{equation}\label{eq:alg13}
\|x^{k+1}-x^{*}\|^{2} \leq \|x^k-\lambda_k F(z^k) -x^{*}\|^{2}+2\gamma_k\|x^{k+1}-x^k \|^2.
\end{equation}
 Using   \eqref{eq:alg12}  and the item $(ii)$ of Lemma~\ref{le:projeccion} we have 
\begin{equation*} 
\|{\PJ}_{H_k}(x^k)- x^*\|^{2}\leq \|x^k-x^{*}\|^{2}-\|{\PJ}_{H_k}(x^k)-x^k\|^2.
\end{equation*}
Since  $x^k\notin {\cal C}^*$,  Proposition~\ref{prop:hk}  implies that  $x^k \notin H_k$. Hence, the last inequality  together with the first equality in   \eqref{eq:etakc} yield
\begin{equation*} 
\|x^k-\lambda_k F(z^k)- x^*\|^{2}\leq \|x^k-x^{*}\|^{2}-\lambda_k^2 \|F(z^k)\|^2.
\end{equation*}
As a result of combining the last inequality with \eqref{eq:alg13}, we arrive at the conclusion that
\begin{equation}\label{eq:alg15}
\| x^{k+1}-x^{*} \|^{2} \leq \|x^k-x^{*}\|^{2}-\lambda_k^2 \|F(z^k)\|^2+2\gamma_k\|x^{k+1}-x^k\|^2.
\end{equation}
Since $x^{k+1} \in {\cal P}^{\gamma_k}_{\cal C}\big(x^{k}, x^k-\lambda_k F(z^k) \big)$,  applying item (ii) of Lemma~\ref{Le:ProjProperty} with $\gamma=\gamma_k$, $x=x^k$, $z=z^k$, $\alpha=\lambda_k$ and $w(\alpha)=x^{k+1}$ we obtain that 
$$
\| x^{k+1}-x^k\|\leq\dfrac{\lambda_k}{1-\gamma_k}\|F(z^k)\|,
$$
which combined with \eqref{eq:alg15} yields
$$
\| x^{k+1}-x^{*} \|^{2} \leq \|x^k-x^{*}\|^{2}-\lambda_k^2 \|F(z^k)\|^2+2\gamma_k\dfrac{\lambda_k^2}{(1-\gamma_k)^2}\|F(z^k)\|^2, 
$$
or equivalently, 
\begin{equation} \label{eq:algf14}
\| x^{k+1}-x^{*} \|^{2} \leq \|x^k-x^{*}\|^{2}-\frac{1}{(1-\gamma_k)^2}\left(\gamma_k^2-4\gamma_k+1\right)\lambda_k^2 \|F(z^k)\|^2.
\end{equation}
Since $0<\bar{\gamma}< 2-\sqrt{3}$,  the function $(0, {\bar \gamma}] \mapsto \left(\gamma_k^2-4\gamma_k+1\right)/(1-\gamma_k)^2$ is increasing  and positive, the inequality \eqref{eq:algf14bg}  follows \eqref{eq:algf14}.  As a consequence,  \eqref{eq:FejerPropertyLs} follows from \eqref{eq:algf14bg} and the proof is  complete.
 \end{proof}   
 The following theorem shows the most important finding about Algorithm~\ref{Alg:ExtraGradMethodLS}.  It establishes the convergence of the sequence towards a solution of VIP(F,${\cal C}$).
\begin{theorem}\label{th:FejerPropertyNLS}
Let ${\cal C}\subset {\mathbb R}^n$  be  a nonempty and closed  convex set and $(x^k)_{k\in{\mathbb N}}$ the sequence generated by Algorithm~\ref{Alg:ExtraGradMethod}. Assume that   $F: {\mathbb R}^n\to {\mathbb R}^n$ is  a    pseudo monotone  operator on  ${\cal C}$ with respect to the solution set  ${\cal C}^*$ of  problem~\eqref{eq:mp}.  If ${\cal C}^*\neq \varnothing$, then Algorithm~\ref{Alg:ExtraGradMethod} either ends at iteration $k$, in which case $x^k\in {\cal C}^*$, or generates an infinite sequence $(x^k)_{k\in{\mathbb N}}$  that converges to a point belonging to ${\cal C}^*$.
\end{theorem} 
\begin{proof}
First, we assume that   Algorithm~\ref{Alg:ExtraGradMethod} ends at iteration $k$. In this case,  we have $y^{k}=x^{k}$ and Remark~\ref{rem:stop} implies that  $x^k\in {\cal C}^*$. Now, we assume that the sequence $(x^k)_{k\in{\mathbb N}}$ is infinite. Hence,   we have  $x^k\notin {\cal C}^*$, for all $k=1, 2, \ldots$.

Since   $(x^k)_{k\in{\mathbb N}}$ satisfies \eqref{eq:FejerPropertyLs} in Lemma~\ref{le:FejerPropertyLS},  it also satisfies Definition~\ref{def:fejer}. Thus, due  to ${\cal C}^*\neq \varnothing$, it follows from item $(ii)$ of Lemma~\ref{le:fejer} that $(x^k)_{k\in{\mathbb N}}$  is bounded.   Using  \eqref{eq:algf14bg} of Lemma~\ref{le:FejerPropertyLS} we have 
\begin{equation} \label{eq:algf14bgls}
0<\frac{1}{(1-{\bar \gamma})^2}\left({\bar \gamma}^2-4{\bar \gamma}+1\right)\lambda_k^2 \|F(z^k)\|^2 \leq \|x^k-x^{*}\|^{2}-\| x^{k+1}-x^{*} \|^{2}, \qquad k=1, 2, \ldots.
\end{equation}
On the other hand, \eqref{eq:FejerPropertyLs} implies that the sequence  $(\|x^k-x^*\|)_{k\in{\mathbb N}}$ is monotone non-increasing and bounded from below. Thus,  $(\|x^k-x^*\|)_{k\in{\mathbb N}}$ converges. Hence, taking  the limit in \eqref{eq:algf14bgls} as $k$ tends to infinity, we have 
$
\lim_{k\to +\infty}\lambda_k \|F(z^k)\|=0.
$
And, in view of  \eqref{eq:etakfs} we conclude that 
\begin{equation} \label{eq:etakmth1}
		\lim_{k\to +\infty}\lambda_k \|F(z^k)\|=\lim_{k\to +\infty} \frac{1}{\|F(z^{k})\|}\big{\langle} F( z^k) ,z^{k}-x^k \big{\rangle}=0.
\end{equation}
Since $y^{k} \in {\cal P}^{\gamma_k}_{\cal C}\big(x^{k}, x^k-\beta_k F(x^k) \big)$,  applying item $(i)$ of Corollary~\ref{cr:ProjPropertyii} with $\gamma=\gamma_k$, $x=x^k$, $\alpha=\beta_k$ and $y=y^{k}$ we obtain that 
$$
\|y^{k}-x^{k}\|\leq\dfrac{\beta_k}{1-\beta_k}\|F(x^k)\|.
$$
Because  $0< {\hat \beta}< \beta_k<{\bar \beta}$, $0 \leq \gamma_k < \bar{\gamma}$ and $(x^k)_{k\in{\mathbb N}}$  is bounded, the latter  inequality implies that $(y^k)_{k\in{\mathbb N}}$  is bounded. Hence, it follows from \eqref{eq:ipzk} that $(z^k)_{k\in{\mathbb N}}$  is also bounded. In addition, due to $F$ be continuous, we conclude that   $(F(z^k))_{k\in{\mathbb N}}$ is bounded.  Thus,  from \eqref{eq:etakmth1}  we have 
$
		\lim_{k\to +\infty} \big{\langle} F( z^k) ,z^{k}-x^k \big{\rangle}=0.
$
Therefore, it follows from the last equality and \eqref{eq:ipzk} that 
\begin{equation} \label{eq:etakmth2}
		\lim_{k\to +\infty} \sigma \alpha^{i_k}\big{\langle} F( z^k) ,y^{k}-x^k \big{\rangle}=0.
\end{equation}
Since  the sequences $(x^k)_{k\in{\mathbb N}}\subset {\cal C}$, $(y^k)_{k\in{\mathbb N}}\subset {\cal C}$ and  $(z^k)_{k\in{\mathbb N}}\subset {\cal C}$ are  bounded, we can take  subsequences  $(x^{k_j})_{j\in{\mathbb N}}$,  $(y^{k_j})_{j\in{\mathbb N}}$ and  $(z^{k_j})_{j\in{\mathbb N}}$ of them, respectively,   and ${\bar x}\in {\cal C}$,  ${\bar y}\in {\cal C}$ and  ${\bar z}\in {\cal C}$  such that $\lim_{j\to+\infty}x^{k_j}={\bar x}$, $\lim_{j\to+\infty}y^{k_j}={\bar y}$ and $\lim_{j\to+\infty}z^{k_j}={\bar z}$. Furthermore,  due to  $0<\alpha<1$, $0 \leq \gamma_k < \bar{\gamma}$ and $0< {\hat \beta}< \beta_k<{\bar \beta}$ for all $k\in{\mathbb N}$, we can also assume without loss of generality that $\lim_{j \to +\infty}  \alpha^{i_{k_j}} = \bar{\alpha} \in [0,1]$, $\lim_{j \to +\infty} \gamma_{k_j} =  {\hat \gamma} \leq {\bar \gamma}$ and  $\lim_{j \to +\infty} \beta_{k_j} ={\tilde \beta}\geq {\hat \beta}$. We have two possibilities for $\bar{\alpha}$: $\bar{\alpha} > 0$ or $\bar{\alpha} = 0$.

 First we assume that $\bar{\alpha} > 0$. In this case, it follows from \eqref{eq:etakmth2} that 
\begin{equation*} 
	0=\lim_{j\to +\infty} { \sigma \alpha^{i_{k_j}}}\big{\langle} F( z^{k_j}) ,y^{k_j}-x^{k_j} \big{\rangle}=\sigma{\bar \alpha}\big{\langle} F({\bar z}) ,{\bar y}-{\bar x} \big{\rangle}.
\end{equation*}
Because we are assuming that ${\bar \alpha} > 0$,  we conclude that $\langle F({\bar z}) ,{\bar y}-{\bar x} \rangle=0$. Using,  \eqref{eq:algjk} and \eqref{eq:ipzk}  together with Proposition~\ref{prop:Bfb} we conclude that 

$$
\big{\langle} F(z^{k_j} ) ,y^{k_j}-x^{k_j} \big{\rangle} \leq \rho \big{\langle} F(x^{k_j} ) ,y^{k_j}-x^{k_j} \big{\rangle}\leq  -\rho \frac{\max \{\rho,\sqrt{3}-1\}}{\bar{\beta}}\|y^{k_j}-x^{k_j} \|^2.
$$ 
Taking  the limit in the previous inequality as $j$ tending to infinity and  taking into account  $\lim_{j\to+\infty}x^{k_j}={\bar x}$, $\lim_{j\to+\infty}y^{k_j}={\bar y}$ and  $\langle F({\bar z}) ,{\bar y}-{\bar x} \rangle=0$, we conclude that  ${\bar y}={\bar x}$. Considering that  $y^{k} \in {\cal P}^{\gamma_k}_{\cal C}\big(x^{k}, x^k-\beta_k F(x^k) \big)$,  it follows from  Definition~\ref{def:InexactProj} that 
\begin{equation*} 
	\big\langle x^{k_j}-\beta_{k_j} F(x^{k_j})-y^{k_j}, y-y^{k_j} \big\rangle \leq  \gamma_{k_j} \|y^{k_j}-x^{k_j}\|^2, \qquad \quad ~\forall y \in {\cal C}.
\end{equation*}
Thus, taking  the limit in the previous inequality as $j$ tending to infinity, using that  ${\bar y}={\bar x}$ and  $\lim_{j \to +\infty} \beta_{k_j} ={\tilde \beta}>0$, we obtain  that 
$$
\left\langle F({\bar x}),  y-{\bar x} \right\rangle\geq 0 , \qquad \forall~y\in {\cal C}, 
$$
which implies that ${\bar x}\in {\cal C}^*$. Since ${\bar x}$ is a cluster point of $(x_k)_{k\in\mathbb{N}}$ and the sequence $(x_k)_{k\in\mathbb{N}}$  is Fej\'er convergent to $ {\cal C}^*$, item (ii) of Lemma~\ref{le:fejer} implies that $\lim_{k\to +\infty}x_{k}={\bar x}$.

Now, let  us assume   that $\bar{\alpha} = 0$. We proceed to prove that in this case,  $(x^k)_{k\in{\mathbb N}}$ likewise converges to some point belonging to the set ${\cal C}^*$. For that,  we consider the auxiliary sequence $({\hat z}_k)_{k\in\mathbb{N}}$  defined by 
\begin{equation} \label{eq:ipzkpf}
					{\hat z}^{k}:=x^{k}+\sigma \frac{\alpha^{i_k}}{\alpha}(y^{k}-x^k) , \qquad  \qquad  k=1,2, \ldots, 
\end{equation}
where $i_{k}$ is defined in \eqref{eq:algjk}. Since $\lim_{j\to+\infty}x^{k_j}={\bar x}$, $\lim_{j\to+\infty}y^{k_j}={\bar y}$ and $\lim_{j \to +\infty} \alpha_{k_j} = \bar{\alpha}=0$, it follows from \eqref{eq:ipzkpf}  that $\lim_{j\to +\infty} {\hat z}^{k_j}={\bar x}$. The definition of ${i_{k_j}}$ in  \eqref{eq:algjk} implies that 
\begin{equation}  \label{eq:onls}
\big{\langle} F {\big(}{\hat z}^{k_j} {\big )},y^{k_j}-x^{k_j}\big{\rangle} > \rho  \big{\langle} F(x^{k_j}), y^{k_j}-x^{k_j}\big{\rangle}.
\end{equation}
Thus,    \eqref{eq:onls}  implies that  $\big{\langle} F({\bar x}) ,{\bar y}-{\bar x} \big{\rangle} \geq \rho \big{\langle} F({\bar x}) ,{\bar y}-{\bar x} \big{\rangle}$, and since $\rho<1$ we conclude that  
\begin{equation} \label{eq:opjkpfn}
\big{\langle} F({\bar x}) ,{\bar y}-{\bar x} \big{\rangle} \geq 0.
\end{equation}
 Given that   $y^{k} \in {\cal P}^{\gamma_k}_{\cal C}\big(x^{k}, x^k-\alpha_k F(x^k) \big)$,  it follows from  Definition~\ref{def:InexactProj} that 
\begin{equation*} 
	\big\langle x^{k_j}-\beta_{k_j} F(x^{k_j})-y^{k_j}, y-y^{k_j} \big\rangle \leq  \gamma_{k_j} \|y^{k_j}-x^{k_j}\|^2, \qquad \quad ~\forall y \in {\cal C}.
\end{equation*}
Taking  the limit in the previous inequality as $j$ going to infinity, and using that  ${\hat \gamma} \leq {\bar \gamma}$, we have
\begin{equation} \label{eq:Projwmthsc} 
 \langle {\bar x}-{\bar y}, y-{\bar y}\rangle -{\tilde \beta} \big\langle F({\bar x}), y-{\bar y}\big\rangle \leq  {\bar \gamma} \|{\bar y}-{\bar x}\|^2, \qquad \quad ~\forall y \in {\cal C}.
\end{equation}
Substituting  $y \in {\cal C}$ for   ${\bar x}\in {\cal C}$ in the last inequality, after some algebraic manipulations   yields 
\begin{equation} \label{eq:Projthscy} 
 {\tilde \beta} \big\langle F({\bar x}), {\bar y}-{\bar x}\big\rangle\leq ({\bar \gamma}-1) \|{\bar y}-{\bar x}\|^2.
\end{equation}
Combining \eqref{eq:opjkpfn} with the latter inequality we obtain  that $(1-{\bar \gamma}) \|{\bar y}-{\bar x}\|^2\leq 0$.  Hence,  because  \eqref{eq:bargamma} implies that $1-{\bar \gamma}>0$,  we conclude that ${\bar y}={\bar x}$. Therefore,  due to ${\tilde \beta}>0$ and ${\bar y}={\bar x}$, it follows from \eqref{eq:Projwmthsc}  that 
$$
\big\langle F({\bar x}), y-{\bar x}\big\rangle \geq  0, \qquad \quad ~\forall y \in {\cal C},
$$
which also implies that ${\bar x}\in {\cal C}^*$. Again, because  ${\bar x}$ is a cluster point of $(x_k)_{k\in\mathbb{N}}$ and the sequence $(x_k)_{k\in\mathbb{N}}$  is Fej\'er convergent to $ {\cal C}^*$, item (ii) of Lemma~\ref{le:fejer} implies that $\lim_{k\to +\infty}x_{k}={\bar x}$ and the proof is concluded.
\end{proof}       

\section{Numerical Results}

To demonstrate the behavior of the proposed algorithms, we present the results of numerical experiments. We implemented Algorithms~\ref{Alg:ExtraGradMethod} and~\ref{Alg:ExtraGradMethodLS} and applied them to test problems adapted from~\cite{buramillan}. For that,  we slightly modified the feasible set ${\cal C}$ to be the unit ball in $\R^d$ with the $p$-norm, defined as   
\begin{equation} \label{eq:np}
B^d_{p}:=\Big\{x:=(x_1, \ldots, x_d)\in \R^d:~ \|x\|_{p}:=\left(\sum_{i=1}^d x_i^{p}\right)^{1/p}\leq 1\Big\}.
\end{equation}
Projections onto this set are more challenging than those on the original problems' feasible sets. The algorithms were implemented in the Julia language. The code can be obtained from~\url{https://github.com/ugonj/extragradient}.

Given the necessity to compute an inexact projection, we now introduce an algorithm designed precisely for this purpose—the well-established classical Frank-Wolfe algorithm (FW algorithm) detailed, for instance, in \cite{BeckTeboulle2004}. Notably, the FW algorithm we present exhibits versatility, enabling the computation of both inexact and exact projections. To achieve the latter, one simply needs to set a tolerance for the exact projection, ensuring it is significantly smaller than the tolerance employed for the inexact projection. For a more thorough guide on computing inexact projections, please consult \cite{LemesPrudente2022}.  To present the FW algorithm, we assume the existence of a linear optimization oracle (referred to as LO oracle) capable of minimizing linear functions over the constraint set ${\cal C}$. The FW algorithm is formally defined to compute a feasible inexact projection $w\in {\cal C}$ of a point $v\in {\mathbb R}^n$ onto a compact convex set ${\cal C}$ with respect to a reference point ${u}\in {\cal C}$ and a forcing parameter $\gamma \geq 0$ as follows:
\begin{algorithm}[h]
\begin{footnotesize}
	\begin{algorithmic}[1]
	\State {Take $\gamma>0$,   ${v}, u\in {\mathbb R}^n$ and ${w_0}\in {\cal C}$. Set  $\ell=0$.}
	\State { Use a LO oracle to compute an optimal solution $z_\ell$ and the optimal value $s_{\ell}^*$ as
\begin{equation}\label{eq:CondGC}
z_\ell := \arg\min_{y \in  C} \,\langle w_\ell-v, ~y-w_\ell\rangle,  \qquad s_{\ell}^*:=\langle  w_\ell-v, ~z_\ell-w_\ell \rangle.
\end{equation}}
	 \State {If $-s^*_{\ell}\leq  \gamma \|{w_\ell}-{u}\|^2$, then {\bf stop}, and  set ${{w}}:=w_\ell$. Otherwise,  set 
\begin{equation}\label{eq:stepsize}
w_{\ell+1}:=w_\ell+ \alpha_\ell(z_\ell-w_\ell), \qquad {\alpha}_\ell: =\min\Big\{1, \frac{-s^*_{\ell}}{\|z_\ell-w_\ell\|^2}  \Big\}.
\end{equation}}
       \State { Set $\ell\gets \ell+1$, and go to Step 2.}
	\end{algorithmic}
\end{footnotesize}
\caption{{FW algorithm} $w\in {\cal P}^{\gamma}_{\cal C}(u, v)$}	
\label{FW:algorithm}
\end{algorithm}

Now, let us show  that the {FW algorithm} described above is capable of computing the feasible inexact projection $w\in {\cal C}$ of a point $v\in {\mathbb R}^n$ onto a compact convex set ${\cal C}$ concerning a reference point ${u}\in {\cal C}$ and a forcing parameter $\gamma \geq 0$. To achieve this, consider $v \in \mathbb{R}^n$ and $\psi_v: \mathbb{R}^n \to \mathbb{R}$ defined by $\psi_v(y) := \|y - v\|^2/2$, with ${\cal C} \subset \mathbb{R}^n$ being a convex compact set. In this case, \eqref{eq:CondGC} is equivalent to $s_{\ell}^* := \min_{y \in \mathcal{C}} \langle \psi_v'(w_\ell),~y-w_\ell\rangle$. Since  $\psi_v$ is convex, we have $\psi_v(y) \geq \psi_v(w_\ell) + \langle \psi_v'(w_\ell),~y-w_\ell\rangle \geq \psi_v(w_\ell) + s_{\ell}^*$ for all $y \in {\cal C}$. Set $w_* := \arg \min_{y \in \mathcal{C}} \psi_v(y)$ and $\psi^* := \psi_v(w_*)$. Letting $y = w_*$ in the last inequality, we have $\psi_v(w_\ell) \geq \psi^* \geq \psi_v(w_\ell) + s_{\ell}^*$, implying that $s_{\ell}^* \leq 0$. Thus, $-s_{\ell}^* = \langle v-w_\ell, ~z_\ell-w_\ell \rangle \geq 0 \geq \langle v-w_*, ~y-w_* \rangle$, for all $y \in {\cal C}$. Therefore, we can set the stopping criterion to $-s_{\ell}^* \leq \gamma \|{w_\ell}-{u}\|^2$, ensuring that the FW algorithm will terminate after finitely many iterations. When the FW algorithm computes $w_\ell \in \mathcal{C}$ satisfying $-s_{\ell}^* \leq \gamma \|{w_\ell}-{u}\|^2$, the method terminates. Otherwise, it computes the stepsize $\alpha_\ell = \arg\min_{\alpha \in [0,1]} \psi_v(w_\ell + \alpha(y_\ell - w_\ell))$ using exact line search. Since $z_\ell$, $w_\ell \in \mathcal{C}$ and $\mathcal{C}$ is convex, we conclude from \eqref{eq:stepsize} that $w_{\ell+1} \in \mathcal{C}$; thus, the FW algorithm generates a sequence in $\mathcal{C}$. Finally, \eqref{eq:CondGC} implies that $\langle v-w_\ell, ~y-w_\ell\rangle \leq -s_{\ell}^*$ for all $z \in \mathcal{C}$. Hence, considering the stopping criterion $-s_{\ell}^* \leq \gamma \|{w_\ell}-{u}\|^2$, we conclude that any output of the FW algorithm ${w} \in \mathcal{C}$ is a feasible inexact projection onto $\mathcal{C}$ of the point $v \in \mathbb{R}^n$ with respect to $u \in \mathbb{R}^n$ and a relative error tolerance function $\gamma \|{w_\ell}-{u}\|^2$, i.e., 
\begin{equation*}
\langle u-{w}, ~y-{w}\rangle \leq \gamma \|{w_\ell}-{u}\|^2 \quad \forall~y\in \mathcal{C}.
\end{equation*}

In the subsequent sections, we present numerical experiments related to Algorithms~\ref{Alg:ExtraGradMethod} and~\ref{Alg:ExtraGradMethodLS}. Consequently, it is imperative to provide nuanced considerations relevant to our numerical results. Firstly, it is crucial to acknowledge the prominent standing of the Extragradient Method in the Variational Inequality literature, characterized by numerous well-established and widely recognized variants known for their efficacy in solving such problems, as evidenced in \cite{Malitsky2018, MalitskySemenov2015, MalitskySemenov2014, SolodovTseng1996}.  Our proposed method introduces an additional paradigm within this domain, explicitly tailored for a well-defined class of problems—specifically, instances where computing an inexact projection proves more efficient than obtaining an exact projection. Our study does not aim to claim universal superiority for our method but rather to offer a specialized solution for the addressed problem class. It intends to demonstrate that, within the realm of projection methods for variational inequality problems utilizing exact projection, there may be advantages in employing inexact projections.  Consequently, our research represents a preliminary step toward the development of inexact projection methods for variational inequality problems based on more sophisticated projection techniques, such as those discussed in \cite{CensorGibaliReich2011, CruzIusem2009, Malitsky2018, Malitsky2020, MalitskySemenov2015, MalitskySemenov2014, SolodovTseng1996, Tseng2000}.  In light of these considerations, it is essential to emphasize that our presentation of numerical experiments serves as an illustrative analysis rather than an exhaustive comparison with all potential variants of the Extragradient Method. Therefore, our focus is solely on the classic Extragradient Method, serving as a benchmark reference.

 \subsection{Lipschitz operator}
 
 In this section, we demonstrate the convergence of Algorithm~\ref{Alg:ExtraGradMethod} for a Lipschitz continuous operator. To do this, we consider the constrained set ${\cal C} = B^2_{10}\subset \R^2$, the unit ball in the $10$-norm as in \eqref{eq:np}, and the operator
    \[
      {\hat T}(x) = \begin{bmatrix}-1&-1\\1&-1\end{bmatrix}x + \begin{bmatrix}3/2\\1/2\end{bmatrix}
    \]
    We applied Algorithm~\ref{Alg:ExtraGradMethod} on VIP(${\hat T}$,${\cal C}$). The iterates are depicted in Figure~\ref{fig:algonols}.

    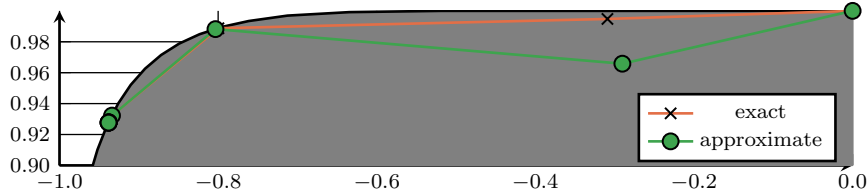
\begin{figure}[H]
      \begin{center}


      \begin{tikzpicture}[/tikz/background rectangle/.style={fill={rgb,1:red,1.0;green,1.0;blue,1.0}, fill opacity={1.0}, draw opacity={1.0}}, show background rectangle]
      \begin{axis}[point meta max={nan}, point meta min={nan}, legend cell align={left}, legend columns={1}, title={Results for exact and inexact projections for $\alpha$=0.31}, title style={at={{(0.5,1)}}, anchor={south}, font={{\fontsize{14 pt}{18.2 pt}\selectfont}}, color={rgb,1:red,0.0;green,0.0;blue,0.0}, draw opacity={1.0}, rotate={0.0}, align={center}}, legend style={color={rgb,1:red,0.0;green,0.0;blue,0.0}, draw opacity={1.0}, line width={1}, solid, fill={rgb,1:red,1.0;green,1.0;blue,1.0}, fill opacity={1.0}, text opacity={1.0}, font={{\fontsize{8 pt}{10.4 pt}\selectfont}}, text={rgb,1:red,0.0;green,0.0;blue,0.0}, cells={anchor={center}}, at={(0.98, 0.02)}, anchor={south east}}, axis background/.style={fill={rgb,1:red,1.0;green,1.0;blue,1.0}, opacity={1.0}}, anchor={north west}, xshift={1.0mm}, yshift={-1.0mm}, scaled x ticks={false}, xlabel={}, x tick style={color={rgb,1:red,0.0;green,0.0;blue,0.0}, opacity={1.0}}, x tick label style={color={rgb,1:red,0.0;green,0.0;blue,0.0}, opacity={1.0}, rotate={0}}, xlabel style={at={(ticklabel cs:0.5)}, anchor=near ticklabel, at={{(ticklabel cs:0.5)}}, anchor={near ticklabel}, font={{\fontsize{11 pt}{14.3 pt}\selectfont}}, color={rgb,1:red,0.0;green,0.0;blue,0.0}, draw opacity={1.0}, rotate={0.0}}, xmajorgrids={true}, xmin={-1}, xmax={0}, xticklabels={{$-1.0$,$-0.8$,$-0.6$,$-0.4$,$-0.2$,$0.0$}}, xtick={{-1.0,-0.8,-0.6000000000000001,-0.4,-0.2,0.0}}, xtick align={inside}, xticklabel style={font={{\fontsize{8 pt}{10.4 pt}\selectfont}}, color={rgb,1:red,0.0;green,0.0;blue,0.0}, draw opacity={1.0}, rotate={0.0}}, x grid style={color={rgb,1:red,0.0;green,0.0;blue,0.0}, draw opacity={0.1}, line width={0.5}, solid}, axis x line*={left}, x axis line style={color={rgb,1:red,0.0;green,0.0;blue,0.0}, draw opacity={1.0}, line width={1}, solid}, scaled y ticks={false}, ylabel={}, y tick style={color={rgb,1:red,0.0;green,0.0;blue,0.0}, opacity={1.0}}, y tick label style={color={rgb,1:red,0.0;green,0.0;blue,0.0}, opacity={1.0}, rotate={0}}, ylabel style={at={(ticklabel cs:0.5)}, anchor=near ticklabel, at={{(ticklabel cs:0.5)}}, anchor={near ticklabel}, font={{\fontsize{11 pt}{14.3 pt}\selectfont}}, color={rgb,1:red,0.0;green,0.0;blue,0.0}, draw opacity={1.0}, rotate={0.0}}, ymajorgrids={true}, ymin={0.9}, ymax={1.0}, yticklabels={{$0.90$,$0.92$,$0.94$,$0.96$,$0.98$}}, ytick={{0.9000000000000001,0.9200000000000002,0.9400000000000002,0.9600000000000002,0.9800000000000002}}, ytick align={inside}, yticklabel style={font={{\fontsize{8 pt}{10.4 pt}\selectfont}}, color={rgb,1:red,0.0;green,0.0;blue,0.0}, draw opacity={1.0}, rotate={0.0}}, y grid style={color={rgb,1:red,0.0;green,0.0;blue,0.0}, draw opacity={0.1}, line width={0.5}, solid}, axis y line*={left}, y axis line style={color={rgb,1:red,0.0;green,0.0;blue,0.0}, draw opacity={1.0}, line width={1}, solid}, colorbar={false}]
          \addplot[color={rgb,1:red,0.0;green,0.0;blue,0.0}, name path={b55856b5-374f-45e9-9ce8-6730891b2575}, area legend, fill={rgb,1:red,0.502;green,0.502;blue,0.502}, fill opacity={0.2}, draw opacity={0.2}, line width={1}, solid, forget plot]
              table[row sep={\\}]
              {
                  \\
                  -1.0  -0.5185495618128161  \\
                  -0.9991867102753419  -0.6345680004017088  \\
                  -0.9969353963973346  -0.7125598220810812  \\
                  -0.9934286234961162  -0.7614187286643554  \\
                  -0.9897166686552987  -0.794396404677748  \\
                  -0.9849005129062742  -0.8233633057575335  \\
                  -0.9788608439055726  -0.8489994717399421  \\
                  -0.9714675187196411  -0.8718230439574342  \\
                  -0.9625934194250519  -0.8921441506451215  \\
                  -0.9521458470341183  -0.9101564093558234  \\
                  -0.940008176821787  -0.926057806251828  \\
                  -0.926057806251828  -0.940008176821787  \\
                  -0.9101564093558234  -0.9521458470341183  \\
                  -0.8921441506451215  -0.9625934194250519  \\
                  -0.8718230439574342  -0.9714675187196411  \\
                  -0.8489994717399421  -0.9788608439055726  \\
                  -0.8233633057575335  -0.9849005129062742  \\
                  -0.794396404677748  -0.9897166686552987  \\
                  -0.7614187286643554  -0.9934286234961162  \\
                  -0.7125598220810812  -0.9969353963973346  \\
                  -0.6345680004017088  -0.9991867102753419  \\
                  -0.5185495618128161  -1.0  \\
                  0.5185495618128161  -1.0  \\
                  0.6345680004017088  -0.9991867102753419  \\
                  0.7125598220810812  -0.9969353963973346  \\
                  0.7614187286643554  -0.9934286234961162  \\
                  0.794396404677748  -0.9897166686552987  \\
                  0.8233633057575335  -0.9849005129062742  \\
                  0.8489994717399421  -0.9788608439055726  \\
                  0.8718230439574342  -0.9714675187196411  \\
                  0.8921441506451215  -0.9625934194250519  \\
                  0.9101564093558234  -0.9521458470341183  \\
                  0.926057806251828  -0.940008176821787  \\
                  0.940008176821787  -0.926057806251828  \\
                  0.9521458470341183  -0.9101564093558234  \\
                  0.9625934194250519  -0.8921441506451215  \\
                  0.9714675187196411  -0.8718230439574342  \\
                  0.9788608439055726  -0.8489994717399421  \\
                  0.9849005129062742  -0.8233633057575335  \\
                  0.9897166686552987  -0.794396404677748  \\
                  0.9934286234961162  -0.7614187286643554  \\
                  0.9969353963973346  -0.7125598220810812  \\
                  0.9991867102753419  -0.6345680004017088  \\
                  1.0  -0.5185495618128161  \\
                  1.0  0.5185495618128161  \\
                  0.9991867102753419  0.6345680004017088  \\
                  0.9969353963973346  0.7125598220810812  \\
                  0.9934286234961162  0.7614187286643554  \\
                  0.9897166686552987  0.794396404677748  \\
                  0.9849005129062742  0.8233633057575335  \\
                  0.9788608439055726  0.8489994717399421  \\
                  0.9714675187196411  0.8718230439574342  \\
                  0.9625934194250519  0.8921441506451215  \\
                  0.9521458470341183  0.9101564093558234  \\
                  0.940008176821787  0.926057806251828  \\
                  0.926057806251828  0.940008176821787  \\
                  0.9101564093558234  0.9521458470341183  \\
                  0.8921441506451215  0.9625934194250519  \\
                  0.8718230439574342  0.9714675187196411  \\
                  0.8489994717399421  0.9788608439055726  \\
                  0.8233633057575335  0.9849005129062742  \\
                  0.794396404677748  0.9897166686552987  \\
                  0.7614187286643554  0.9934286234961162  \\
                  0.7125598220810812  0.9969353963973346  \\
                  0.6345680004017088  0.9991867102753419  \\
                  0.5185495618128161  1.0  \\
                  -0.5185495618128161  1.0  \\
                  -0.6345680004017088  0.9991867102753419  \\
                  -0.7125598220810812  0.9969353963973346  \\
                  -0.7614187286643554  0.9934286234961162  \\
                  -0.794396404677748  0.9897166686552987  \\
                  -0.8233633057575335  0.9849005129062742  \\
                  -0.8489994717399421  0.9788608439055726  \\
                  -0.8718230439574342  0.9714675187196411  \\
                  -0.8921441506451215  0.9625934194250519  \\
                  -0.9101564093558234  0.9521458470341183  \\
                  -0.926057806251828  0.940008176821787  \\
                  -0.940008176821787  0.926057806251828  \\
                  -0.9521458470341183  0.9101564093558234  \\
                  -0.9625934194250519  0.8921441506451215  \\
                  -0.9714675187196411  0.8718230439574342  \\
                  -0.9788608439055726  0.8489994717399421  \\
                  -0.9849005129062742  0.8233633057575335  \\
                  -0.9897166686552987  0.794396404677748  \\
                  -0.9934286234961162  0.7614187286643554  \\
                  -0.9969353963973346  0.7125598220810812  \\
                  -0.9991867102753419  0.6345680004017088  \\
                  -1.0  0.5185495618128161  \\
                  -1.0  -0.5185495618128161  \\
              }
              ;
          \addplot[color={rgb,1:red,0.8889;green,0.4356;blue,0.2781}, name path={bef6139f-2705-4ecd-866a-c561e9337812}, draw opacity={1.0}, line width={1}, solid, mark={x}, mark size={3.0 pt}, mark repeat={1}, mark options={color={rgb,1:red,0.0;green,0.0;blue,0.0}, draw opacity={1.0}, fill={rgb,1:red,0.8889;green,0.4356;blue,0.2781}, fill opacity={1.0}, line width={0.75}, rotate={0}, solid}]
              table[row sep={\\}]
              {
                  \\
                  0.0  1.0  \\
                  -0.30915216912587085  0.9947756756580887  \\
                  -0.7996403735488538  0.9887588297908424  \\
                  -0.9336714457052196  0.9323905810637686  \\
                  -0.9378352707060426  0.9279973908494805  \\
                  -0.9380977711031966  0.9277079867358344  \\
                  -0.9381168417516134  0.9276869013823932  \\
                  -0.9381182646921065  0.9276853277887518  \\
                  -0.9381183827576298  0.9276851972210528  \\
              }
              ;
          \addlegendentry {exact}
          \addplot[color={rgb,1:red,0.2422;green,0.6433;blue,0.3044}, name path={a0cf5b5d-fcbd-40b1-8170-3e8f3d21ec22}, draw opacity={1.0}, line width={1}, solid, mark={*}, mark size={3.0 pt}, mark repeat={1}, mark options={color={rgb,1:red,0.0;green,0.0;blue,0.0}, draw opacity={1.0}, fill={rgb,1:red,0.2422;green,0.6433;blue,0.3044}, fill opacity={1.0}, line width={0.75}, rotate={0}, solid}]
              table[row sep={\\}]
              {
                  \\
                  0.0  1.0  \\
                  -0.2903182259918812  0.9658163821742148  \\
                  -0.8033638612527966  0.9881946446993284  \\
                  -0.9338224178738042  0.9322375077289164  \\
                  -0.937850132923969  0.9279810465136427  \\
                  -0.9380993345626893  0.927706258413234  \\
                  -0.9381168839030448  0.9276868547688529  \\
                  -0.9381182691467759  0.927685322862375  \\
                  -0.9381183828620131  0.927685197105616  \\
              }
              ;
          \addlegendentry {approximate}
      \end{axis}
      \end{tikzpicture}
      \caption{Convergence of Algorithm~\ref{Alg:ExtraGradMethod} with exact and inexact projections. On this problem the number of steps is the same in both cases.}\label{fig:algonols}
      \end{center}
    \end{figure}

   Table~\ref{tab:resultslipschitz} shows how the Algorithm~\ref{Alg:ExtraGradMethod}  performed for various values of $\alpha$ and ${\bar \gamma}$. The third columns represent the number of steps taken by the extragradient algorithm to reach the solution, and the last column represents the total number of linear searches applied by the Frank-Wolfe method.  It is interesting to compare the case when  ${\bar \gamma}$ is small (say, ${\bar \gamma}=0.01$) to the cases when ${\bar \gamma}$ is larger: it can be seen that although the algorithm takes the same number of steps to reach the solution, when ${\bar \gamma}\approx 0$ (in the case when the projection is almost exact), significantly more Frank-Wolfe iterations were performed. In other words, since performing approximate projections does not increase the total number of steps in the extragradient method, it is beneficial to use them, as each step of this method requires less work.

    \begin{table}[ht]
      \centering
        \begin{tabular}{rrrr}
          \toprule
          {$\mathbf{\alpha}$} & {$\mathbf{{\bar \gamma}}$} & {N. steps} & {N. linear searches} \\\midrule
          0.01 & 0.01 & 109 & 1.0317e7 \\
          0.01 & 0.106 & 109 & 6.62431e6 \\
          0.01 & 0.49 & 109 & 6.57129e6 \\\midrule
          0.11 & 0.01 & 16 & 12997 \\
          0.11 & 0.106 & 16 & 1085 \\
          0.11 & 0.394 & 16 & 966 \\\midrule
          0.21 & 0.01 & 11 & 2444 \\
          0.21 & 0.106 & 11 & 237 \\
          0.21 & 0.394 & 11 & 239 \\\midrule
          0.31 & 0.01 & 9 & 935 \\
          0.31 & 0.106 & 9 & 129 \\
          0.31 & 0.298 & 9 & 126 \\\midrule
          0.41 & 0.01 & 9 & 476 \\
          0.41 & 0.106 & 9 & 113 \\\bottomrule
        \end{tabular}
      \caption{Behaviour of the algorithm for various values of $\alpha$ and ${\bar \gamma}$.}\label{tab:resultslipschitz}
    \end{table}

        \subsection{Non-Lipschitz operator}
        
        In this section we consider the constrained set ${\cal C} = B^2_{10}\subset \R^2$, the unit ball in the $10$-norm as in \eqref{eq:np}, and the operator ${\tilde T}(x_1,x_2) = -\frac{t(x_1,x_2)}{1+t(x_1,x_2)}(1,1)$, where $t(x_1,x_2)=(x_1+\sqrt{x_1^2+4x_2})/2$. The operator ${\tilde T}$ is quasimonotone, and it is pseudomonotone with respect to the solution set, as it has a unique solution $(1,1)/\|(1,1))\|_{10}$, see \cite{buramillan}. However it is not Lipschitz. We applied Algorithm~\ref{Alg:ExtraGradMethodLS} on VIP(${\tilde T}$,${\cal C}$), with initial point at $(0,1)$. The iterates are depicted in Figure~\ref{fig:algols}.

    \begin{figure}[H]
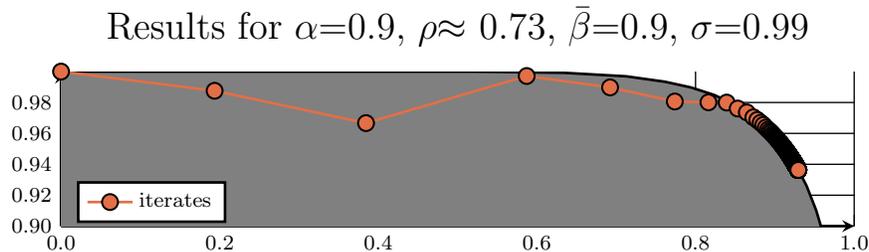

      \begin{center}



        \caption{Convergence of Algorithm~\ref{Alg:ExtraGradMethodLS}.}\label{fig:algols}
        \end{center}
    \end{figure}

    \subsection{Higher dimensions}
    In this section, we apply Algorithm~\ref{Alg:ExtraGradMethod} to a problem in $d$-dimensional Euclidean space $\R^d$, which is a modified version of~\cite[Example 5.3]{buramillan}. Consider the unit norm ball in $p$-norm $B^d_{p}$,  as defined in \eqref{eq:np},  and let   $T_{h}(x) := (T_{h,1}(x),…,T_{h,d}(x))$ be  an operators,   where
\[
T_{h,i}(x) := \frac{h x_i \sum_{j=1}^{d}x_j - \frac{1}{2}h\sum_{j=1}^{d}x_j^2 - 1}{(\sum_{i=1}^d x_i)^2},
\]
and $h$ is a value such that $0.1 \leq h \leq 1.6$.  The operator $T_h$ vanishes at the point $x=\alpha e$, where $\alpha=(2/(dh))^{0.5}$ and $e:=(1,1,\ldots,1)\in \R^d$. When this point is in the set $B^d_{p}$, then it is a solution to the problem $VIP(T,B^d_{p})$. Otherwise, when $h\geq 2d^{(\frac{2}{{p}}-1)}$, a solution is the Euclidean projection of  $x=\alpha e$ onto $B^d_{p}$, i.e.,  $x^* = {\PJ}_{B^d_{p}}(\alpha e)$.

In all experiments we use the starting point $x^1=(0,\ldots,0,1)$.

\subsubsection{In $\R^5$}

We first consider the convergence when $d=5$. In these experiments we set $h=0.6$. The convergence of the sequence of iterates is shown on Figure~\ref{fig:5dims}. It can be seen that the sequence converges rapidly towards the solution of the problem, which is $\nicefrac{e}{\|e\|_{{p}}}\in \R^5$.

\begin{figure}[H]
\begin{tikzpicture}[/tikz/background rectangle/.style={fill={rgb,1:red,1.0;green,1.0;blue,1.0}, fill opacity={1.0}, draw opacity={1.0}}, show background rectangle]
\begin{axis}[point meta max={nan}, point meta min={nan}, legend cell align={left}, legend columns={1}, title={Convergence of the sequence $x^{k}$}, title style={at={{(0.5,1)}}, anchor={south}, font={{\fontsize{14 pt}{18.2 pt}\selectfont}}, color={rgb,1:red,0.0;green,0.0;blue,0.0}, draw opacity={1.0}, rotate={0.0}, align={center}}, legend style={color={rgb,1:red,0.0;green,0.0;blue,0.0}, draw opacity={1.0}, line width={1}, solid, fill={rgb,1:red,1.0;green,1.0;blue,1.0}, fill opacity={1.0}, text opacity={1.0}, font={{\fontsize{8 pt}{10.4 pt}\selectfont}}, text={rgb,1:red,0.0;green,0.0;blue,0.0}, cells={anchor={center}}, at={(0.98, 0.98)}, anchor={north east}}, axis background/.style={fill={rgb,1:red,1.0;green,1.0;blue,1.0}, opacity={1.0}}, anchor={north west}, xshift={1.0mm}, yshift={-1.0mm},  height={0.5\figurewidth}, scaled x ticks={false}, xlabel={$k$}, x tick style={color={rgb,1:red,0.0;green,0.0;blue,0.0}, opacity={1.0}}, x tick label style={color={rgb,1:red,0.0;green,0.0;blue,0.0}, opacity={1.0}, rotate={0}}, xlabel style={at={(ticklabel cs:0.5)}, anchor=near ticklabel, at={{(ticklabel cs:0.5)}}, anchor={near ticklabel}, font={{\fontsize{11 pt}{14.3 pt}\selectfont}}, color={rgb,1:red,0.0;green,0.0;blue,0.0}, draw opacity={1.0}, rotate={0.0}}, xmajorgrids={true}, xmin={0.16000000000000014}, xmax={29.84}, xticklabels={{$5$,$10$,$15$,$20$,$25$}}, xtick={{5.0,10.0,15.0,20.0,25.0}}, xtick align={inside}, xticklabel style={font={{\fontsize{8 pt}{10.4 pt}\selectfont}}, color={rgb,1:red,0.0;green,0.0;blue,0.0}, draw opacity={1.0}, rotate={0.0}}, x grid style={color={rgb,1:red,0.0;green,0.0;blue,0.0}, draw opacity={0.1}, line width={0.5}, solid}, axis x line*={left}, x axis line style={color={rgb,1:red,0.0;green,0.0;blue,0.0}, draw opacity={1.0}, line width={1}, solid}, scaled y ticks={false}, ylabel={}, y tick style={color={rgb,1:red,0.0;green,0.0;blue,0.0}, opacity={1.0}}, y tick label style={color={rgb,1:red,0.0;green,0.0;blue,0.0}, opacity={1.0}, rotate={0}}, ylabel style={at={(ticklabel cs:0.5)}, anchor=near ticklabel, at={{(ticklabel cs:0.5)}}, anchor={near ticklabel}, font={{\fontsize{11 pt}{14.3 pt}\selectfont}}, color={rgb,1:red,0.0;green,0.0;blue,0.0}, draw opacity={1.0}, rotate={0.0}}, ymajorgrids={true}, ymin={-0.04744272052555276}, ymax={1.6925152743338832}, yticklabels={{$0.0$,$0.5$,$1.0$,$1.5$}}, ytick={{0.0,0.5,1.0,1.5}}, ytick align={inside}, yticklabel style={font={{\fontsize{8 pt}{10.4 pt}\selectfont}}, color={rgb,1:red,0.0;green,0.0;blue,0.0}, draw opacity={1.0}, rotate={0.0}}, y grid style={color={rgb,1:red,0.0;green,0.0;blue,0.0}, draw opacity={0.1}, line width={0.5}, solid}, axis y line*={left}, y axis line style={color={rgb,1:red,0.0;green,0.0;blue,0.0}, draw opacity={1.0}, line width={1}, solid}, colorbar={false}]
    \addplot[color={rgb,1:red,0.0;green,0.6056;blue,0.9787}, line width={1}, solid]
        table[row sep={\\}]
        {
            \\
            1.0  1.6432711801397482  \\
            2.0  1.5114712865582025  \\
            3.0  1.348907424075164  \\
            4.0  1.169262298979628  \\
            5.0  0.9639041135926758  \\
            6.0  0.7596562285979648  \\
            7.0  0.5652207539627455  \\
            8.0  0.44136699792001965  \\
            9.0  0.35464618813190535  \\
            10.0  0.2896768699912317  \\
            11.0  0.23916258141484723  \\
            12.0  0.19895513735933112  \\
            13.0  0.16643323120425793  \\
            14.0  0.1398209744968307  \\
            15.0  0.11785462997917928  \\
            16.0  0.09960167927955013  \\
            17.0  0.08435471967562556  \\
            18.0  0.07156542366669304  \\
            19.0  0.06080143323005738  \\
            20.0  0.05171710903338949  \\
            21.0  0.04403302868652776  \\
            22.0  0.0375212193517286  \\
            23.0  0.0319942702128082  \\
            24.0  0.02729714324645106  \\
            25.0  0.02330090622269261  \\
            26.0  0.019897864389267196  \\
            27.0  0.0169977291826715  \\
            28.0  0.01452456878475269  \\
            29.0  0.012414356996706177  \\
        }
        ;
    \addlegendentry {$\|x^{k}-x^*\|$}
    \addplot[color={rgb,1:red,0.8889;green,0.4356;blue,0.2781}, line width={1}, solid]
        table[row sep={\\}]
        {
            \\
            1.0  0.13180224807229451  \\
            2.0  0.16398697298285878  \\
            3.0  0.1853924257048514  \\
            4.0  0.2114681916949644  \\
            5.0  0.20568444727730648  \\
            6.0  0.19495849420097194  \\
            7.0  0.12410878707326718  \\
            8.0  0.08684884657616979  \\
            9.0  0.0650366398555692  \\
            10.0  0.05055104833048798  \\
            11.0  0.04022809187798979  \\
            12.0  0.03253375445520416  \\
            13.0  0.026619168340435935  \\
            14.0  0.0219704286638999  \\
            15.0  0.01825538891606851  \\
            16.0  0.015248427267326722  \\
            17.0  0.012790185404776507  \\
            18.0  0.010764532380046148  \\
            19.0  0.009084655928039708  \\
            20.0  0.007684284172251106  \\
            21.0  0.006511934966049534  \\
            22.0  0.005527026778422415  \\
            23.0  0.004697175053686203  \\
            24.0  0.003996266863154986  \\
            25.0  0.003403060378769406  \\
            26.0  0.0029001467480342215  \\
            27.0  0.0024731675887135003  \\
            28.0  0.002110216272520039  \\
            29.0  0.0018013736685821805  \\
        }
        ;
    \addlegendentry {$\|x^{k+1}-x^{k}\|$}

\end{axis}
\end{tikzpicture}
\caption{Convergence of the iterates of the proposed method.}\label{fig:5dims}
\end{figure}
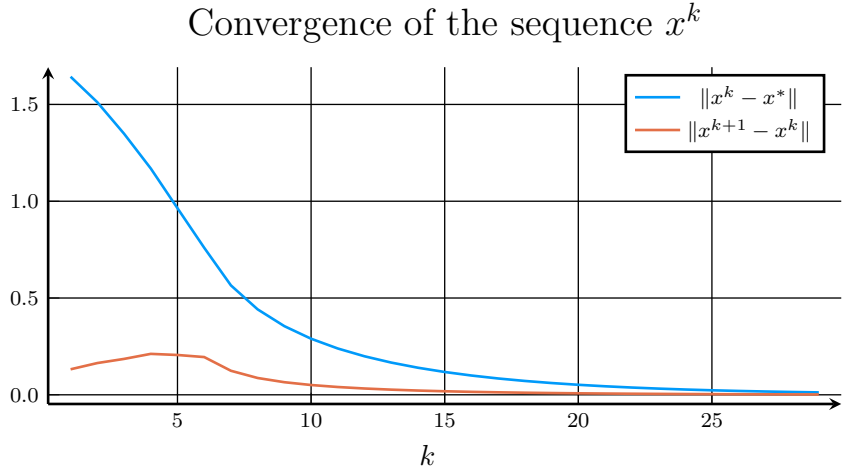

\subsubsection{Higher Dimensions}

We also observe the number of iterations required by our algorithm to reach an approximate solution. We vary  $d$ of the $d$-dimensional Euclidean space $\R^d$ and $p$ of the  ${p}$-norm. We set $h=0.2$ so that in most problems that we solve the solution to the VIP problem is not a zero of the operator $T_{h}$. The number of iterations, for various values of $d$ and ${p}$ is provided in Table~\ref{tab:highdim}.

\begin{table}[ht]
  \begin{tabular}{*{12}{r}}
    \toprule
     &$d$ & \textbf{5} & \textbf{6} & \textbf{7} & \textbf{8} & \textbf{9} & \textbf{10} & \textbf{11} & \textbf{12} & \textbf{13} & \textbf{14} \\
     $p$ &&&&&&&&&&& \\\midrule
    10    &         & 704        & 926    
    & 1083       & 1644       & 2045       & 2502        & 2783        & 3371        & 3578        & 3643        \\
    15     &       & 283        & 419        & 481        & 804        & 976        & 1227        & 1484        & 1564        & 1249        & 256         \\\bottomrule
  \end{tabular}

  \begin{tabular}{*{11}{r}}
    \toprule
     &$d$ & \textbf{15} & \textbf{16} & \textbf{17} & \textbf{18} & \textbf{19} & \textbf{20} & \textbf{25} & \textbf{50} & \textbf{100} \\
     $p$&&&&&&&&&&\\\midrule
                10 &&        3193 &        2291 &         863 &         308 &         316 &         325 &         364 &         519 &                  736 \\
                15 &&         281 &         290 &         299 &         308 &         317 &         325 &         365 &         518 &       735 \\\bottomrule
  \end{tabular}
  \caption{Number of iterations required to obtain an approximate solution (such that $\|x_k-x^*\|\leq 10^{-2}$).}
  \label{tab:highdim}
\end{table}
    
 \section{Conclusions} \label{Sec:Conclusions}
   We investigate the Extragradient method for solving variational inequality problems using inexact projections onto the feasible set in this paper. We expect that our study can contribute to further research on the subject, particularly in solving large-scale problems where the computing effort of each iteration is connected with projections into the feasible set. Indeed, the idea of employing inexactness in the projection rather than exactness is very attractive from a computational perspective. It is worth noting that the Frank-Wolfe approach has a low computing cost for each iteration, resulting in great computational performance in various types of compact sets, as cited in  \cite{GarberHazan2015, Jaggi2013}. Searching for new efficient methods, such as the Frank-Wolfe method which generates inexact projections, is a subject that calls for attention.

   \backmatter
   \bmhead{Acknowledgements}

 The first and last authors were supported by the Australian Research Council (ARC),  Solving hard Chebyshev approximation problems through nonsmooth analysis (Discovery Project DP180100602). The second author was supported in part by  CNPq grant 304666/2021-1.

  \section{Data availability and conflict of interest statement}

 We do not analyse or generate any datasets, because our work proceeds within a theoretical and mathematical approach. All code used to produce the results of the numerical experiments is available at \url{https://github.com/ugonj/extragradient}. The authors declare no conflict of interest.

\bibliographystyle{plain}
\bibliography{InexactExtraGradMethodVIP}
\end{document}